\documentclass[12pt, reqno]{amsart}

\usepackage[a4paper,top=3cm,bottom=2.5cm,left=2.75cm,right=2.75cm,marginparwidth=1.75cm]{geometry}

\usepackage{amsmath}
\usepackage{amsthm}
\usepackage{amsfonts}
\usepackage{amssymb}
\usepackage{mathtools}
\usepackage[mathscr]{euscript}
\usepackage{bm}
\usepackage[table]{xcolor}
\usepackage[all,cmtip]{xy}
\usepackage{enumitem}
\usepackage{tikz}  
\usepackage{tikz-cd}
\usepackage{graphicx}
\usepackage{scalerel}
\usepackage{microtype}
\usepackage{comment}
\usepackage{etoolbox}

\patchcmd{\abstract}{\global\let\abstract\relax}{\global\let\abstract\relax
  \par\addvspace{\baselineskip}
  \noindent\footnotesize\textbf{Mathematics Subject Classification (2020):} \@subjclass
  \par\addvspace{0.5\baselineskip}
  \noindent\footnotesize\textbf{Keywords:} \@keywords
}{}{}

\usepackage{verbatim}
\usepackage[colorlinks]{hyperref}
\hypersetup{
    colorlinks,
    citecolor=magenta,
    filecolor=red,
    linkcolor=blue,
    urlcolor=purple,
    pdftitle={The Verra Involution on Special Cubic Fourfolds},
    pdfauthor={Augustinas Jacovskis, Reinder Meinsma and Riccardo Moschetti},
}

\usepackage[T2A,T1]{fontenc}
\DeclareSymbolFont{cyrillic}{T2A}{cmr}{m}{n}
\DeclareMathSymbol{\Sha}{\mathalpha}{cyrillic}{216}

\usepackage{comment}
\usepackage{array}
\usepackage{tikz-cd}

\definecolor{citation}{rgb}{0,.40,.80}
\definecolor{reference}{rgb}{.80,0,.40}
\definecolor{todogray}{HTML}{E4E4E4}

\numberwithin{equation}{section}

\theoremstyle{plain}
\newtheorem{theorem}{Theorem}[section]
\newtheorem{lemma}[theorem]{Lemma}
\newtheorem{proposition}[theorem]{Proposition}
\newtheorem{corollary}[theorem]{Corollary}
\newtheorem{conjecture}[theorem]{Conjecture}
\newtheorem{question}[theorem]{Question}

\theoremstyle{definition}
\newtheorem{definition}[theorem]{Definition}
\newtheorem{example}[theorem]{Example}
\newtheorem{remark}[theorem]{Remark}

\newtheorem*{remark*}{Remark}

\newtheorem{warning}[theorem]{Warning}

\newcommand{\set}[1]{\left\{ \, #1 \, \right\}}
\newcommand{\cFM}{\cF\cM}

\newcommand{\Db}{\mathrm{D^b}}

\DeclareMathOperator{\Pic}{Pic}
\DeclareMathOperator{\NS}{NS}

\DeclareMathOperator{\Spec}{Spec}

\newcommand{\functionstar}[4]{
\begin{array}{rcl} #1 &\longrightarrow &#2 \\ #3&\longmapsto &#4 \end{array}
}

\newcommand{\isomorphism}[5]{\begin{array}{lrcl} 
#1: & #2 & \overset{\sim}{\longrightarrow} & #3 \\
    & #4 & \longmapsto & #5 \end{array}
}

\newcommand{\function}[5]{\begin{array}{lrcl} 
#1: & #2 & \longrightarrow & #3 \\
    & #4 & \longmapsto & #5 \end{array}
}

\DeclareMathOperator{\Hom}{Hom}
\DeclareMathOperator{\End}{End}
\DeclareMathOperator{\Ext}{Ext}
\DeclareMathOperator{\Aut}{Aut}

\DeclareMathOperator{\Bl}{Bl}
\DeclareMathOperator{\dv}{div}
\DeclareMathOperator{\res}{res}

\newcommand{\id}{\mathrm{id}}

\newcommand{\verlat}{\ensuremath{\Lambda_{V_4}^0}}
\newcommand{\gravc}{\ensuremath{K_0(\operatorname{Var}_\bC)}}
\newcommand{\gghs}{\ensuremath{K_0^{\oplus}(\operatorname{HS}_{\bZ,2})}}

\newcommand{\gghsfour}{\ensuremath{K_0^{\oplus}(\operatorname{HS}_{\bZ,4})}}

\DeclareMathOperator{\FM}{FM}

\DeclareMathOperator{\Jac}{J}

\newcommand{\klat}{\ensuremath{\Lambda_{\operatorname{K3}}}}

\DeclareMathOperator{\Br}{Br}

\newcommand{\disc}{\mathrm{disc}}

\newcommand{\rk}{\mathrm{rk}}

\docsvlist{A,B,C,D,E,F,G,H,I,J,K,L,M,N,O,P,Q,R,S,T,U,V,W,X,Y,Z} 

\docsvlist{A,B,C,D,E,F,G,H,I,J,K,L,M,N,O,P,Q,R,S,T,U,V,W,X,Y,Z} 

\newcommand{\bA}{\mathbf{A}}
\newcommand{\bC}{\mathbf{C}}

\newcommand{\bZ}{\mathbf{Z}}
\newcommand{\bP}{\mathbf{P}}
\newcommand{\bQ}{\mathbf{Q}}

\newcommand{\bL}{\mathbf{L}}

\usepackage[numbers]{natbib}
\begin{document}

\title[The Verra involution on special cubic fourfolds]{The Verra involution on  special cubic fourfolds}

\author[Augustinas Jacovskis, Reinder Meinsma and Riccardo Moschetti]{Augustinas Jacovskis, Reinder Meinsma and Riccardo Moschetti}
\address{Jacovskis: Nacionalinis Fizinių ir Technologijos Mokslų Centras, 
Saulėtekio al. 3, Vilnius, 10257, Lithuania}
\email{\texttt{augustinas.jacovskis@ff.vu.lt}}
\address{Meinsma: Fakultät für Mathematik und Informatik,
Universität des Saarlandes,
Campus E2.4, 66123 Saarbrücken, Germany
}
\email{\texttt{meinsma@math.uni-sb.de}}
\address{Moschetti: Department of Mathematics F. Casorati, University of Pavia, via Ferrata 5, 27100 Pavia, Italy} 
\email{\texttt{riccardo.moschetti@unipv.it}}
\begin{abstract}
    We present a detailed study of the Hodge theory, geometry and derived categories of Verra fourfolds, and their relationship with twisted K3 surfaces and cubic fourfolds. We classify Brauer classes of Verra type on K3 surfaces of degree 2, and use this to derive an explicit counting formula for Fourier--Mukai partners of a large class of Verra fourfolds. We develop the basic theory of cohomology of Verra fourfolds and apply it to prove a derived Torelli theorem for a large class of Verra fourfolds. As an application, we prove that L-equivalent Verra fourfolds in this class are Fourier--Mukai partners, and conjecture that this implication holds for all Verra fourfolds. We then give a construction of cubic fourfolds birational to Verra fourfolds. The construction yields examples of cubic and Verra fourfolds whose associated twisted K3 surfaces have non-trivial Brauer classes. Therefore, these fourfolds are expected to be irrational.
\end{abstract}

\maketitle

\tableofcontents

\section{Introduction}
\addtocontents{toc}{\protect\setcounter{tocdepth}{1}}
\renewcommand{\thetheorem}{\Alph{theorem}}

The interplay between derived categories and birational geometry is a central theme in the study of higher-dimensional algebraic varieties. 
Cubic fourfolds provide a natural testing ground for this interplay, because their categorical, Hodge-theoretic, and birational geometry are closely linked via their relationship with K3 surfaces and associated hyperk\"ahler manifolds. 

In particular, special cubic fourfolds lying in certain Noether--Lefschetz divisors (called \textit{Hassett divisors}) admit additional structures that allow one to recover a twisted K3 surface, both lattice-theoretically and categorically. 

An important example of this is given by a cubic fourfold $X \subset \bP^5$ containing a plane $P \subset \bP^5$. Given such a pair $(X,P)$, the projection to the plane $P$ induces a quadric surface bundle 
$\Bl_PX \to P$, whose discriminant locus is a sextic curve $C \subset \bP^2$. To such a quadric surface bundle, we can associate a twisted K3 surface $(S,\alpha)$ via the Stein factorisation of the moduli space of lines contained in fibres $\cF(\Bl_PX/P) \to P$. Moreover, the Kuznetsov component of $X$ is given by $\cA_X \simeq \Db(S,\alpha)$ \cite[Theorem 4.3]{Kuzcub} (see also Example \ref{ex:cubic fourfolds quadric fibration}). 

Since a quadric surface bundle over $\bP^2$ is rational as soon as its associated twisted K3 surface has trivial Brauer class, this example led Kuznetsov to conjecture that a cubic fourfold is rational if and only if its Kuznetsov component is equivalent to the derived category of a K3 surface \cite[Conjecture 1.1]{Kuzcub}.
Later, Huybrechts conjectured \cite[Conjecture 2.5]{HuybrechtsUpdate} that two cubic fourfolds are birational whenever their Kuznetsov components are equivalent (see Conjecture \ref{conj: Huybrechts birational categorical Torelli}), and Brooke, Frei and Marquand conjectured \cite[Conjecture 1.2]{BFM} that cubic fourfolds with birational Fano varieties of lines are birational (see Conjecture \ref{conj: BFM}).

We note that the rationality question for very general cubic fourfolds was recently settled \cite{KKPY}. Nevertheless, Kuznetsov's, Huybrechts' and Brooke--Frei--Marquand's conjectures are still open. 

In this paper, we study two types of quadric surface bundles: cubic fourfolds containing a plane, and Verra fourfolds. 

Let $V_4$ be a \textit{Verra fourfold}, that is, a double cover $V_4 \to \bP^2 \times \bP^2$ branched over a divisor $V_3 \subset \bP^2\times \bP^2$ of bidegree $(2,2)$. The divisor $V_3$ is known as an \emph{ordinary Verra threefold}. Verra fourfolds, like cubic fourfolds, are Fano varieties of \textit{K3-type}. The Kuznetsov component $\cA_{V_4}$ is equivalent to the derived category of a twisted K3 surface $\Db(S,\alpha)$, which is analogous to the case of a cubic fourfold containing a plane. 
We note that the rationality question for very general Verra fourfolds is still open, though there has been some recent progress \cite{BGMP, FayVerra}.

The conjectures of Huybrechts and Kuznetsov have natural translations to the setting of Verra fourfolds. However, we note that no examples of non-isomorphic Verra fourfolds with equivalent Kuznetsov components have been constructed before the present paper.

\subsection{Our results}

We say that a twisted K3 surface $(S,\alpha)$ is \textit{of Verra type} if there is a Verra fourfold $V_4$, one of whose associated twisted K3 surfaces is isomorphic to $(S,\alpha)$, see Definition \ref{def: verra type}. Twisted K3 surfaces of Verra type are our main object of study.

For a twisted K3 surface $(S,\alpha)$, the property of being of Verra type is equivalent to the existence of a double cover $\varphi \colon S \to \bP^2$ with branch locus $C \subset \bP^2$, and a non-trivial $2$-torsion Brauer class $\beta \in \Br(\bP^2 \setminus C)_2$ such that $\varphi^*\beta=\alpha$ \cite[\S9]{vGe}, \cite[Lemma 5.9]{JM}. In this case, the $2$-torsion Brauer class $\beta$ corresponds to the conic bundle $V_3 \to \bP^2$, where $V_3$ is the branch locus of the double cover $V_4 \to \bP^2 \times \bP^2$.

\subsubsection{Fourier--Mukai partners of Verra fourfolds}

In our first main result, we determine the kernel and cokernel of the map $\Br(\bP^2 \setminus C)_2 \to \Br(S)_2$ precisely. This is built on the work of Ford \cite{Ford} and Ingalls--Obus--Ozman--Viray \cite{IOOV}.

\begin{theorem}[=Theorem \ref{thm: kernel and cokernel of res}] \label{thm: kernel and cokernel of res INTRO}
     Let $S$ be a K3 surface of Picard rank $\rho$ with an ample divisor $H$ of degree $H^2 = 2$. Let $C\subset \bP^2$ be the branch curve of the double cover $S \to \bP^2$. There is an exact sequence
    \[
    0\to (\bZ/2\bZ)^{\rho-r} \to \Br(\bP^2\setminus C)_2 \overset{\res}{\longrightarrow} \Br(S)_2 \to (\bZ/2\bZ)^{2-r}\to 0,
    \]
    where 
    \[
    r = \begin{cases}
        1 & \text{ if } \dv_{\NS(S)}(H) = 2 \\
        2 & \text{ if } \dv_{\NS(S)}(H) = 1.
    \end{cases}
    \]    
\end{theorem}

It follows from Theorem \ref{thm: kernel and cokernel of res INTRO} that if $S$ has Picard rank $2$, and $\dv(H) = 1$, then every non-trivial $2$-torsion Brauer class $\alpha \in \Br(S)_2$ is of Verra type.
Using Theorem \ref{thm: kernel and cokernel of res INTRO}, combined with a careful analysis of the cohomology of Verra fourfolds in Section \ref{sec: cohomology of Verra fourfolds} we produce a counting formula for Fourier--Mukai partners of Verra fourfolds whose associated twisted K3 surfaces have these properties. Here, we say that two Verra fourfolds $V_4$ and $V_4'$ are Fourier--Mukai partners if their Kuznetsov components $\cA_{V_4}$ and $\cA_{V_4'}$ are Fourier--Mukai equivalent.

\begin{theorem}[=Theorem \ref{thm: counting formula}]\label{thm: counting formula INTRO}
   Let $V_4$ be a Verra fourfold with associated twisted polarised K3 surface $(S,H,\alpha)$. Suppose $\rho(S)=2$ and $\dv(H) = 1$. Moreover, assume $V_4$ is very general with these properties. Then the number of Fourier--Mukai partners of $V_4$ is given by
   \[
   |\FM(V_4)| = |\mathrm{Amp}_2(S)| \cdot |O(\NS(S)) \backslash O(A_{\NS(S)})/O_{\mathrm{Hodge}}(T(S))|,
   \]
   where $\mathrm{Amp}_2(S)$ is the set of isomorphism classes of double covers $S \to \bP^2$.
\end{theorem}

More generally, we prove the following for arbitrary Verra fourfolds: 
\begin{enumerate}
    \item For any Verra fourfold $V_4$, there are finitely many isomorphism classes of Fourier--Mukai partners of $V_4$ (see Theorem \ref{thm: finitely many FM partners}).
    \item Let $N > 0$ be a natural number. Then there exists a Verra fourfold with at least $N$ Fourier--Mukai partners (see Corollary \ref{cor: unbounded}).
\end{enumerate}

Moreover, we prove the following (derived) Torelli theorem. Here, $T(V_4)$ denotes the transcendental lattice of the Verra fourfold $V_4$, and $H^4_\mathrm{prim}(V_4,\bZ)$ denotes the integral primitive cohomology of $V_4$. Both of these Hodge structures are defined in Section \ref{sec: cohomology of Verra fourfolds}.

\begin{theorem}[=Theorem \ref{thm: derived torelli for verra fourfolds} and Proposition \ref{prop: torelli for very general verra with transcendental}]\label{thm: derived torelli INTRO}
    Let $V_4$ be a Verra fourfold with associated twisted polarised K3 surfaces $(S_1,H_1,\alpha_1)$ and $(S_2,H_2,\alpha_2)$, and let $V_4'$ be a Verra fourfold.
    \begin{enumerate}
        \item Suppose $\rho(S_1) = 2$ and $\dv(H_1) = 1$. Then there is an equivalence $\cA_{V_4}\simeq \cA_{V_4'}$ if and only if there is a Hodge isometry $T(V_4) \simeq T(V_4')$.
        \item On the other hand, if $\rho(S_1) = 1$, then there is an isomorphism $V_4 \simeq V_4'$ if and only if there is a Hodge isometry $H^4_\mathrm{prim}(V_4,\bZ) \simeq H^4_\mathrm{prim}(V_4',\bZ)$.
    \end{enumerate}
\end{theorem}

We then apply Theorem \ref{thm: derived torelli INTRO} to the study of $\bL$-equivalence for Verra fourfolds, and prove that $\bL$-equivalence implies Fourier--Mukai partnership for a large class of Verra fourfolds (see Theorem \ref{thm: L implies D for Verra of type I}). We also show that $\bL$-equivalence classes of Verra fourfolds are finite (see Theorem \ref{thm: L equivalence classes are finite}).

\subsubsection{Cubic fourfolds of Verra type}

We say that a pair $(X,P)$, consisting of a cubic fourfold $X$ containing a plane $P \subset X$, is \textit{of Verra type} if the associated twisted K3 surface $(S,\alpha)$ is of Verra type. 
A priori, it is not clear that such a pair $(X,P)$ exists. Namely, if $\rho(S) = 1$, then $(S,\alpha)$ is never of Verra type (see Proposition \ref{prop: generally not of verra type}). Therefore, to find cubic fourfolds of Verra type, we need to assume $\rho(S) \geq 2$. 

Assume that we have a pair $(X,P)$ such that the algebraic part $A(X)$ of $H^4(X,\bZ)$ has rank $\rk(A(X)) = 3$. Then there exist $\tau, n$ such that the underlying lattice structure of $A(X)$ is determined by the Gram matrix \cite[Corollary 8.14]{YY}
\[
M_{\tau,n} \coloneqq \begin{pmatrix}
    3 & 1 & 0 \\ 1 & 3 & \tau \\ 0 & \tau & 2n
\end{pmatrix}.
\]
Moreover, the matrices $M_{\tau,n}$ define non-empty, irreducible divisors $\cC_{M_{\tau,n}}$ of $\cC_8$, provided $0\leq \tau \leq 4$, $n\geq 2$, and $(\tau,n)\neq (3,2),(4,2),(4,3)$.
We show in Proposition \ref{prop: verra type criterion} that such a pair $(X,P)$ is of Verra type if and only if $\tau = 2$, thus producing infinitely many divisors of $\cC_8$ consisting of cubic fourfolds of Verra type. 

\begin{theorem}[=Theorem \ref{thm: cubic fourfolds of verra type are birational to verra}] \label{thm: cubic fourfolds of verra type are birational to verra INTRO}
    Let $(X,P)$ be a cubic fourfold of Verra type, and let $(S,\alpha)$ be the associated twisted K3 surface. Then there exists a Verra fourfold $V_4$, one of whose associated K3 surfaces is also $(S,\alpha)$, which is birational to $X$:
    \[
    X\overset{\sim}\dashrightarrow V_4.
    \]
    Moreover, we have $\cA_{X}\simeq \cA_{V_4}.$
    \footnote{We were informed by Federica Galluzzi and Bert van Geemen that 
a result analogous to Theorem \ref{thm: cubic fourfolds of verra type are birational to verra INTRO} will appear in the independent forthcoming work \cite{GvG26}.}
\end{theorem}

As a point of interest, Kuznetsov's conjecture predicts that very general members of $\cC_{M_{2,n}}$ are irrational for certain values of $n$, see \cite[Corollary 8.14(3)]{YY}. The existence of irrational cubic fourfolds which are birational to Verra fourfolds is implicitly an open question in \cite{BGMP}, see especially Remark 22 in \textit{loc. cit.}

\subsubsection{The Verra involution on special cubic fourfolds}

Finally, we construct an involution on very general members of the divisors $\cC_{M_{2,n}}$ for $n > 2$ even, such that $4(n-1)+1$ is a power of a prime. We call this involution the \textit{Verra involution}. It is constructed using Theorem \ref{thm: cubic fourfolds of verra type are birational to verra INTRO}, combined with the fact that
Verra fourfolds admit two natural quadric fibrations given by the projections onto the two copies of $\bP^2$.

The Verra involution is constructed explicitly as follows.

\begin{theorem}[=Theorem \ref{thm: cubic correspondence}] \label{thm: cubic correspondence INTRO}
    Let $(X,P)$ be a cubic fourfold of Verra type, and suppose $X$ is very general with this property, and such that $\disc(T(X))$ is a power of a prime number. Let $(S,H,\alpha)$ be the twisted polarised K3 surface associated with the quadric surface bundle $\Bl_P X \to P$, and let $V_4$ be a Verra fourfold, one of whose associated twisted polarised K3 surfaces is also $(S,H,\alpha)$. Let $(S',H',\alpha')$ be the other twisted polarised K3 surface associated with $V_4$, and assume $\alpha'$ is non-trivial. Then there exists a unique cubic fourfold $X'$ containing a plane $P'$ such that the twisted polarised K3 surface associated with the quadric surface bundle $\Bl_{P'} X' \to P'$ is $(S',H',\alpha')$:
    \begin{equation*}
        \xymatrix{
        \Bl_PX \ar[dr] \ar@{-->}[rr]^\simeq && V_4 \ar[dl] \ar[dr] && \Bl_{P'}X' \ar[dl] \ar@{-->}[ll]_\simeq \\
        & \bP^2 && \bP^2 &
        }
    \end{equation*}
\end{theorem}

The Verra involution is defined by the assignment $X \mapsto X'$
where $X$ and $X'$ are as in Theorem \ref{thm: cubic correspondence}. It is only well-defined for very general special cubic fourfolds in certain divisors of $\cC_8$.
It satisfies the following properties:

\begin{theorem}[=Remark \ref{rmk: not isomorphic unless n is 2}, Proposition \ref{prop: properties of Verra involution}, Theorem \ref{thm: FM count for CM2n}] \label{thm: the verra involution INTRO}
\leavevmode
    \begin{enumerate}
        \item The Verra involution is the identity if and only if $n=2$.
        \item The cubic fourfolds $X$ and $X'$ are birational.
        \item The cubic fourfolds $X$ and $X'$ are Fourier--Mukai partners.
        \item The Fano varieties of lines $F(X)$ and $F(X')$ are birational. In particular, $F(X)$ and $F(X')$ are Fourier--Mukai partners.
    \end{enumerate}
\end{theorem}

We note that other involutions on moduli spaces of cubic fourfolds with similar properties have been constructed in \cite{FL23,BM26}. Those involutions are birational maps on the moduli spaces, while our involution is only well-defined away from a countable collection of closed subsets.

\subsection*{Acknowledgements}

We warmly thank Bert van Geemen and Federica Galluzzi for interesting and enlightening discussions and interest in our work, as well as for helpful comments on an earlier draft of this paper. We are also very grateful to Moritz Hartlieb for interesting discussions that eventually led us down the road to the Verra involution.
AJ was supported by Research Council of Lithuania
Grant No. P-PD-25-081. This work is a contribution to Project A22 in SFB 195 No. 286237555 of DFG. The third named author is a member of the INdAM group GNSAGA.
\setcounter{theorem}{0}
\renewcommand{\thetheorem}{\thesection.\arabic{theorem}}

\section{Preliminaries}

\subsection{Lattices} Our main reference for lattice theory is \cite{Nik80}.

A \textit{lattice} is a free-finitely generated abelian group $L$ with a non-degenerate symmetric integral bilinear form $b \colon L \times L \to \bZ$. For $u, v \in L$, we usually denote $b(u,v)$ by $u\cdot v$ and $b(v,v)$ by $v^2$. A lattice $L$ is \textit{even} is $v^2$ is even for all $v \in L$. 
The \textit{dual} of $L$ is defined as
\[
L^\vee \coloneqq \Hom(L,\bZ) \simeq \set{x \in L_\bQ \mid \forall v \in L \colon x\cdot v \in \bZ}.
\]
We have a natural map $L \hookrightarrow L^\vee$, which is injective because $L$ is assumed to be non-degenerate. The quotient $A_L \coloneqq L^\vee/L$ is called the \textit{discriminant group} of $L$. The discriminant group $A_L$ is finite, and has order
\[
\disc(L) \coloneqq |A_L| = |\det(M)|,
\]
where $M$ is any Gram matrix for $L$. The number $\disc(L)$ is called the \textit{discriminant} of $L$. Lattices of discriminant $1$ are said to be \textit{unimodular}.

Since there is a natural quadratic form $q\colon L^\vee \to \bQ$ inherited from $L$, the discriminant group also inherits a quadratic form
\[
q\colon A_L \to \bQ/\bZ.
\]
If $L$ is even, then $q$ takes values in $\bQ/2\bZ$ instead.

A \textit{lattice embedding} or \textit{metric morphism} is a group homomorphism $f\colon N \to L$, where $N$ and $L$ are lattices, satisfying $f(u)\cdot f(v) = u\cdot v$ for all $u,v \in N$. All lattice embeddings are injective by our assumption that lattices are non-degenerate, and the cokernel $L/N$ is a finitely generated abelian group. Lattice embeddings are generally studied in the extremal cases where $L/N$ is either free or finite.

\begin{definition}\label{def: overlattice primitive embedding}
    A lattice embedding $f\colon N \to L$ is called:
    \begin{enumerate}
        \item a \textit{primitive embedding} if $L/N$ is free;
        \item an \textit{overlattice} if $L/N$ is finite;
        \item an \textit{isometry} if $L/N = 0$.
    \end{enumerate}
\end{definition}

For a lattice $L$, the group of isometries $L \simeq L$ is denoted $O(L)$. Similarly, a group isomorphism $A_L \simeq A_L$ respecting the quadratic form on $A_L$ is called an isometry, and the group of isometries of $A_L$ is denoted $O(A_L)$. Note that any isometry $f\colon L \simeq L$ induces an isometry $\overline{f}\colon A_L \simeq A_L$, and that this procedure defines a group homomorphism 
\[
\functionstar{O(L)}{O(A_L)}{f}{\overline f.}
\]

Consider an overlattice $N \hookrightarrow L$. There is a chain of finite-index embeddings
\[
N \hookrightarrow L \hookrightarrow L^\vee \hookrightarrow N^\vee,
\]
which allows us to view the group $L/N$ as a subgroup of $A_N$. Moreover, this group determines the embedding $N\hookrightarrow L$ up to isomorphism:

\begin{lemma}[{\cite[Proposition 1.4.2]{Nik80}}]  \label{lem: overlattice subgroup correspondence}
    Let $N\hookrightarrow L$ and $N'\hookrightarrow L'$ be even overlattices, and define the quotient $H_L \coloneqq L/N$. Then an isometry $f \colon N\simeq N'$ can be extended to an isometry $\tilde{f}\colon L\simeq L'$ if and only if $$\overline{f}(H_L) = H_{L'}.$$ 
\end{lemma}

\subsection{Fourier--Mukai equivalences}

From now on, for $X$ a scheme, $\Db(X)$ will denote the \emph{bounded derived category of coherent sheaves} on $X$.

\begin{definition} \label{def: FM transform}
    Let $X$ and $Y$ be schemes. A functor $\Phi_{\cE} \colon \Db(X) \to \Db(Y)$ given by $\Phi_{\cE}(F) \coloneqq p_{Y *}(p_X^*(F) \otimes \cE)$ is called a \emph{Fourier--Mukai} functor with \emph{Fourier--Mukai kernel} $\cE \in \Db(X \times Y)$, where $p_X$ and $p_Y$ are the obvious projections.
\end{definition}

\subsection{Semiorthogonal decompositions}

Let $\cC$ be a triangulated category. A subcategory $\cA \subset \cC$ is called \emph{admissible} if both the left and right adjoint to the inclusion $\cA \hookrightarrow \cC$ exist. An object $E \in \cC$ is called \emph{exceptional} if $\Ext^k(E,E) = \bC$ if $k=0$ and vanishes if $k \neq 0$.

\begin{definition}
    We say $\cC = \langle \cA_1, \dots , \cA_n \rangle$ is a \emph{semiorthogonal decomposition} of a triangulated category $\cC$ if
    \begin{enumerate}
        \item $\Ext^k(F, G) = 0$ for all $k \in \bZ$ and for all $F \in \cA_i, G \in \cA_j$ if $i > j$;
        \item for any $F \in \cC$, there exists a sequence of morphisms
        \begin{equation*}
        0 = F_n \to F_{n-1} \to \cdots \to F_1 \to F_0 = F
        \end{equation*}
        such that $\mathrm{Cone}(F_i \to F_{i-1}) \in \cA_i$.
    \end{enumerate}
    When the subcategories $\cA_i$ are all exceptional objects satisfying condition (1) above, we call $\{ \cA_1, \dots , \cA_n \}$ an \emph{exceptional collection}. An exceptional collection in a category is called \emph{full} if it generates the category. 
\end{definition}

We also have a notion of Fourier--Mukai functor for semiorthogonal components of derived categories:

\begin{definition} \label{def: fourier--mukai type}
    Let $X$ and $Y$ be schemes, such that their derived categories admit semiorthogonal components $\cA_X$ and $\cA_Y$, respectively. Then a functor $\cA_X \to \cA_Y$ is said to be a \emph{Fourier--Mukai} functor if the composition 
    \begin{equation*}
        \Db(X) \to \cA_X \to \cA_Y \to \Db(Y)
    \end{equation*}
    is a Fourier--Mukai functor between the derived categories. Here, the map $\Db(X) \to \cA_X$ is the projection to $\cA_X$ and the map $\cA_Y \to \Db(Y)$ is the inclusion.
\end{definition}

We say that $X$ admits a \emph{rectangular Lefschetz decomposition} if 
\begin{equation*}
    \Db(X) = \langle \cA, \cA \otimes L, \dots, \cA \otimes L^m \rangle,
\end{equation*}
where $\cA \subset \Db(X)$ is an admissible subcategory, $L$ is a line bundle on $X$, and $m$ is some integer. 

If $X$ is a quadric fibration, then its derived category admits a particular semiorthogonal decomposition:

\begin{theorem}[{\cite[Theorem 4.2]{KuzQuad}}] \label{thm:SOD quad fibration}
    Let $p \colon X \to B$ be a flat quadric fibration of relative dimension $n-2$. Then there is a semiorthogonal decomposition
    \begin{equation*}
        \Db(X) = \langle \Db(B, \cB_0), p^* (\Db(B)) \otimes \cO_{X/B}(1), \dots, p^* (\Db(B)) \otimes \cO_{X/B}(n-2) \rangle .
    \end{equation*}
\end{theorem}

Here, $\Db(B, \cB_0)$ denotes the bounded derived category of sheaves of $\cB_0$-modules on the base $B$, where $\cB_0$ is the sheaf of the even part of the Clifford algebra on $B$. For more details on (sheaves of) Clifford algebras, and for the proof of the above theorem, we point the reader to \cite{KuzQuad}.

\subsection{Kuznetsov components}

In this subsection, we define one of the main objects of study in this paper -- the Kuznetsov component of a given variety. We in particular remark on a few examples of Theorem \ref{thm:SOD quad fibration}, which we will use throughout the paper.

\begin{example}[Cubic fourfolds] \label{ex:cubic fourfolds quadric fibration}
    Recall the usual semiorthogonal decomposition of a cubic fourfold $X \subset \bP^5$
    \begin{equation*}
        \Db(X) = \langle \cA_X, \cO_X, \cO_X(1), \cO_X(2) \rangle,
    \end{equation*}
    where the \emph{Kuznetsov component} $\cA_X$ of the cubic fourfold is defined to be the right orthogonal to to the exceptional collection $\{ \cO_X, \cO_X(1), \cO_X(2) \}$. 
    
    Now assume that $X$ contains a plane $P$. Then $p \colon \Bl_P X \to P$ defines a quadric fibration, and its discriminant curve is a sextic $C$ \cite[Lemma 4.1]{Kuzcub}. Theorem \ref{thm:SOD quad fibration} then gives a semiorthogonal decomposition
    \begin{equation*}
        \Db(\Bl_{P} X) = \langle \Db(P, \cB_0), p^* \Db(P), p^* \Db(P) \otimes \cO_{\Bl_P X}(1) \rangle ,
    \end{equation*}
    where we have omitted the notation for the embedding of $\Db(P, \cB_0)$ into $\Db(\Bl_P X)$. Writing an alternative semiorthogonal decomposition for $\Db(\Bl_P X)$ using Orlov's Blow-up Formula \cite[Theorem 4.3]{OrlovBlowUp} and shuffling one of the decompositions appropriately using mutations, one can show \cite[Theorem 4.3]{Kuzcub} that we have 
    \begin{equation*}
        \cA_X \simeq \Db(P, \cB_0) .
    \end{equation*}
    Finally, consider the moduli space $\cF(\Bl_P X/ P)$ of lines contained in fibres of $\Bl_P X$ over $P$. When the fibre is a smooth quadric, the Fano variety of lines in the quadric is a disjoint union of two copies of $\bP^1$. On the other hand, when the quadric fibre is singular, the Fano variety of lines is a single copy of $\bP^1$. Now let $S$ be the double cover of $P$ branched in $C$. In particular, $S$ is a K3 surface. By the above discussion, we have a $\bP^1$-fibration $\cF(\Bl_P X/ P) \to S$, and thus we get an element $\alpha \in \Br(S)_2$. Crucially, by \cite[Lemma 4.2]{Kuzcub} we have 
    \begin{equation*}
        \Db(P, \cB_0) \simeq \Db(S, \alpha) .
    \end{equation*}
    Note that for the above to work, the quadric fibres of $\Bl_P X \to P$ have to have rank $\geq 3$. The cubic fourfolds that we work with when we apply this example later in the paper are all sufficiently general within their moduli for this to hold. 
\end{example}

\begin{example}[Verra fourfolds] \label{ex: Verra fourfold SOD}
    Since Verra fourfolds admit (two) quadric surface bundle structures over $\bP^2$, by Theorem \ref{thm:SOD quad fibration} we have the following semiorthogonal decomposition, associated to one of the quadric surface bundles:
    \begin{equation*}
        \Db(V_4) = \langle \Db(\bP^2, \cB_0), \Db(\bP^2), \Db(\bP^2) \otimes \cO_{V_4/\bP^2}(1) \rangle ,
    \end{equation*}
    This semiorthogonal decomposition defines the \emph{Kuznetsov component} $\cA_{V_4} \coloneqq \Db(\bP^2, \cB_0)$ of $V_4$. Note that we have the equivalence
    \begin{equation*}
        \Db(\bP^2, \cB_0) \simeq \Db(S, \alpha).
    \end{equation*}
     The analogously defined Kuznetsov component arising from the other quadric fibration is equivalent to the one coming from the first (see also Lemma \ref{lem: both Verra kuznetsov components equivalent}).
\end{example}

\begin{example}[Verra threefolds]
    Ordinary Verra threefolds admit two conic bundle structures over $\bP^2$. Similarly to Verra fourfolds, by Theorem \ref{thm:SOD quad fibration} we thus get a semiorthogonal decomposition 
    \begin{equation*}
        \Db(V_3) = \langle \Db(\bP^2, \cB_0) , \Db(\bP^2)  \rangle .
    \end{equation*}
   This semiorthogonal decomposition defines the \emph{Kuznetsov component} $\cA_{V_3} \coloneqq \Db(\bP^2, \cB_0)$ of $V_3$. 
\end{example}

Note that in all of the examples above, $\cA_{V_4} \not\simeq \cA_{V_3} \not\simeq \cA_{X}$, since the sheaves of Clifford algebras on the respective bases are different.

\section{Brauer classes of Verra type on K3 surfaces}\label{sec: brauer classes of verra type}

\subsection{The Brauer group of a K3 surface} \label{sec: Brauer groups of K3 surfaces}

In this subsection, we briefly recall the definition and properties of the Brauer group of a K3 surface. We follow \cite[Section 1]{GvG}. Let $S$ be a K3 surface. The \emph{Brauer group} $\Br(S)$ of $S$ is defined as 
\begin{equation*}
    \Br(S) = H^2(S, \cO_S^*)_{\mathrm{tors}} .
\end{equation*}
In this paper, we mostly consider Brauer classes of order $2$. 
Via the exponential sequence, an order $2$ class $\alpha \in \Br(S)_2$ has a lift $\widetilde{\alpha} \in H^2(S, \cO_S)$, such that $2 \widetilde{\alpha} \in H^2(S, \bZ) / \Pic (S)$. Any class $B_\alpha \in \frac{1}{2} H^2(S, \bZ) \subset H^2(S, \bQ)$ which is mapped to $\widetilde{\alpha}$ in the exponential sequence is called a \emph{$B$-field lift} of $\alpha$.

The self-duality of the lattice $H^2(S, \bZ)$ along with the isomorphism $$\Br(S) \simeq (H^2(S, \bZ)/ \Pic(S)) \otimes \bQ / \bZ$$ gives an isomorphism $\Br(S)_2 \simeq \Hom(T(S), \bZ / 2 \bZ)$, where for $\alpha \in \Br(S)_2$ we have
\begin{equation*}
    \alpha \colon T(S) \to \frac{1}{2}\bZ / \bZ \simeq \bZ/2\bZ, \qquad x \mapsto x \cdot B_\alpha \pmod \bZ,
\end{equation*}
where $B_\alpha$ is a $B$-field lift of $\alpha$. This interpretation of Brauer classes of $S$ also allows us to define the \emph{twisted transcendental lattice} of $S$ as 
    \begin{equation} \label{eq: def of twisted transcendetnal lattice}
        T(S, \alpha) \coloneqq \ker (\alpha) .
    \end{equation}
The twisted transcendental lattice $T(S,\alpha)$ inherits a Hodge structure as well as a lattice structure from $T(S)$.

\subsection{Brauer classes of Verra type}

In this section, we study certain Brauer classes on K3 surfaces of degree 2.
Let $S$ be a K3 surface with ample divisor $H\in \NS(S)$ of degree $H^2 = 2$. Note that there are no assumptions on the Picard rank of $S$. Then $H$ defines a map $\varphi_H\colon S\to \bP^2$. The map $\varphi_H$ is a double cover branched in a sextic curve $C\subset \bP^2$. Denote $\varphi_H^{-1}(C) \coloneqq R\subset S$ for the ramification curve. 
There is a natural restriction map
\[
\res \colon \Br(\bP^2\setminus C)\to \Br(S\setminus R).
\]
Recall the purity exact sequence:
\begin{lemma}[{\cite[Theorem 3.7.2]{CSbook}}, Purity exact sequence]\label{lem: purity}
Let $X$ be a smooth variety and $D$ a smooth irreducible divisor. Then we have the exact sequence
\[0 \to \Br(X) \to \Br(X \setminus D) \xrightarrow{\partial_D} H^1(D,\bQ/\bZ) . \]
\end{lemma}

By Lemma \ref{lem: purity}, we may view $\Br(S)$ as a subgroup of $\Br(S\setminus R)$. By the same token, we have
\[
\Br(\bP^2\setminus C)\simeq H^1(C,\bQ/\bZ) \simeq (\bQ/\bZ)^{20}.
\]

\begin{lemma}
    Let $\beta \in \Br(\bP^2 \setminus C)_2$. Then $\res(\beta)$ is contained in the subgroup $\Br(S)_2$ of $\Br(S \setminus R)$.
\end{lemma}
\begin{proof}
Consider the purity exact sequences
$$\begin{tikzcd}
0 \arrow[r] & \Br(\bP^2)\arrow[r] & \Br(\bP^2 \setminus C)\arrow[r, "\partial_C"]\arrow[d, "\res"] & H^1(C,\bQ/\bZ) \\
0 \arrow[r] & \Br(S)\arrow[r] & \Br(S \setminus R)\arrow[r, "\partial_R"] & H^1(R,\bQ/\bZ) \\
\end{tikzcd}
$$
Denote by $\varphi_R$ the restriction of $\varphi_H$ to $R$. Since $\varphi_H^*(C)=2R$, the compatibility with the pullback gives
$$\partial_R(\res \beta) = 2 \varphi_R^* (\partial_C(\beta)).$$
Since $\beta$ is $2$-torsion, then also $\partial_C(\beta)$ is $2$-torsion. As a consequence $2 \varphi_R(\res \beta) = 0$, and hence $\partial_R( \res \beta)=0$.

As a consequence $\res(\beta)$ is in the kernel of $\partial_R$, which is the image of $\Br(S)$. Since $\res$ is a group homomorphism and $2 \beta = 0$, we have $2 \res(\beta) = 0$, hence $\res(\beta) \in \Br(S)_2$.
\end{proof}

\begin{definition}[{\cite[Definition 5.8]{JM}}] \label{def: verra type} \leavevmode
    Let $S$ be a K3 surface.
    \begin{enumerate}
        \item Let $H$ be a polarisation of $S$ of degree $2$, and let $C \subset \bP^2$ be the branch locus of the double cover $\varphi_H\colon S \to \bP^2$. A Brauer class $\alpha\in \Br(S)_2$ is said to be \textit{of Verra type for $H$} if there exists a non-trivial Brauer class $1\neq\beta\in \Br(\bP^2\setminus C)_2$ such that $\res(\beta)=\alpha$.
        \item A Brauer class $\alpha \in \Br(S)_2$ is said to be \textit{of Verra type} if there exists a degree 2 polarisation $H\in \NS(S)$ such that $\alpha$ is of Verra type for $H$.
        \item A twisted polarised K3 surface $(S,H,\alpha)$ is said to be \textit{of Verra type} if $\alpha$ is a Brauer class of Verra type for $H$.
        \item A twisted K3 surface $(S,\alpha)$ is said to be \textit{of Verra type} if there exists a degree 2 polarisation $H\in \NS(S)$ such that $(S,H,\alpha)$ is of Verra type.
    \end{enumerate} 
\end{definition}

Before we study the map $\res$, let us explain our terminology. 
Suppose that $\alpha\in \Br(S)_2$ is a Brauer class of Verra type, and 
$\beta \in \Br(\bP^2\setminus C)_2$ satisfies $\res(\beta) = \alpha$. 
Then the Brauer class $\beta$ defines a conic bundle $V_\beta\to \bP^2$ 
with discriminant locus $C\subset \bP^2$, which is also the branch locus of the 
double cover $\varphi_H\colon S\to \bP^2$. 
The variety $V_\beta$ is a Verra threefold, i.e. a bidegree $(2,2)$ divisor in $\bP^2 \times \bP^2$ \cite[\S 9.8]{vGe}. This means that Brauer classes of Verra type are precisely those which arise as pullbacks of Brauer classes of Verra threefolds.

Throughout this paper, we will mostly be concerned with Brauer classes of Verra type. 
More specifically, for a Brauer class 
$\alpha\in \Br(S)_2$ of Verra type, we will study the number of Brauer classes 
$\beta\in \Br(\bP^2\setminus C)_2$ with the property that $\res(\beta)=\alpha$.
We note that, under the assumption that $S$ has Picard rank $1$, there is a short exact sequence \cite{Ford}
\[
0 \to \Br(\bP^2\setminus C)_2\to \Br(S)_2\to \bZ/2\bZ \to 0.
\]
In particular, there is a unique Brauer class $\beta \in \Br(\bP^2\setminus C)_2$ arising as the preimage of any Brauer class $\alpha\in \Br(S)_2$ of Verra type. We now turn to the main result of this section, which shows that this is not the case for K3 surfaces of higher Picard rank. 

\begin{theorem} \label{thm: kernel and cokernel of res}
    Let $S$ be a K3 surface of Picard rank $\rho$ with an ample divisor $H$ of degree $H^2 = 2$. Let $C\subset \bP^2$ be the branch curve of the double cover $S \to \bP^2$. There is an exact sequence
    \[
    0\to (\bZ/2\bZ)^{\rho-r} \to \Br(\bP^2\setminus C)_2 \overset{\res}{\longrightarrow} \Br(S)_2 \to (\bZ/2\bZ)^{2-r}\to 0,
    \]
    where 
    \[
    r = \begin{cases}
        1 & \text{ if } \dv_{\NS(S)}(H) = 2 \\
        2 & \text{ if } \dv_{\NS(S)}(H) = 1.
    \end{cases}
    \]
\end{theorem}

We note that the fact that $\ker(\res) \simeq (\bZ/2\bZ)^{\rho-r}$ for some $1\leq r \leq 2$ was proved by Ford in \cite{Ford}. Our contribution is the criterion that determines the precise value of $r$.

\subsection{Theta characteristics of the branch curve}
Let $S$ be a K3 surface with ample divisor $H\in \NS(S)$ of degree $H^2 = 2$, and let $C\subset \bP^2$ be the branch sextic of the double cover $\varphi_H\colon S\to \bP^2$. 

Under the assumption that $S$ has Picard rank $1$, the 2-torsion subgroup of $\Br(S)$ is isomorphic to the group \cite[\S9]{vGe}:
\begin{equation}\label{eq: pic kc 2}
\left(\frac{\Pic(C)}{\bZ K_C}\right)_2.
\end{equation}
We do not make assumptions on the Picard rank of $S$, but the group \eqref{eq: pic kc 2} is nevertheless important for our discussion.

Firstly, we note that there is a natural injective group homomorphism
\begin{equation}\label{eq: jacobian in pic 2}
\functionstar{\Jac(C)_2}{\left(\frac{\Pic(C)}{\bZ K_C}\right)_2}{\quad[\cL]}{\overline{[\cL]}.}
\end{equation}
Recall that a \textit{theta characteristic} of $C$ is a divisor class 
$\Theta \in \Pic(C)$ such that $2\Theta = K_C$. 
The cokernel of \eqref{eq: jacobian in pic 2} is cyclic of order $2$, and it is
generated by the class of any theta characteristic $\overline{[\Theta]}\in \left(\frac{\Pic(C)}{\bZ K_C}\right)_2$.

Let $L\subset \bP^2$ be a general line, so that $C\cap L$ consists of $6$ distinct points. With a slight abuse of notation, we denote the divisor $C\cap L$ on $C$ by $L$. By the adjunction formula, we have
\[
K_C = 3L.
\]
Then the quotient map is an isomorphism:
\[
\left(\frac{\Pic(C)}{\bZ K_C}\right)_2 \simeq \left(\frac{\Pic(C)}{\bZ L}\right)_2.
\]
For any class $\overline{[\cL]}\in \left(\frac{\Pic(C)}{\bZ L}\right)_2$, we have $\cL^{\otimes2} \simeq \cO(\varepsilon L)$ for some $\varepsilon \in \bZ$. This defines a degree map
\begin{equation}\label{eq: degree map}
\function{\deg}{\left(\frac{\Pic(C)}{\bZ L}\right)_2}{\bZ/2\bZ}{\overline{[\cL]}}{\begin{cases}
0 & \text{ if } \varepsilon \text{ is even} \\
1 & \text{ if } \varepsilon \text{ is odd}.
\end{cases}}
\end{equation}
The kernel of this degree map is $\ker(\deg) = \Jac(C)_2$. In other words, we have a short exact sequence
\[
0 \to \Jac(C)_2 \to \left(\frac{\Pic(C)}{\bZ L}\right)_2 \overset{\deg}{\to} \bZ/2\bZ \to 0.
\]
That means that the set of classes of theta characteristics in $\left(\frac{\Pic(C)}{\bZ L}\right)_2$ is precisely the set of classes $\overline{[\cL]}$ such that $\deg(\overline{[\cL]})=1$.

\subsection{The $2$-torsion Brauer classes on K3 surfaces of degree 2}
Let us recall the following special case of \cite[Theorem 1.1]{IOOV}.
\begin{theorem}[{\cite[Theorem 1.1]{IOOV}}] \label{thm: IOOV main theorem}
    Let $S$ be a K3 surface, and let $H\in \NS(S)$ be an ample divisor of degree $H^2 = 2$. Let $C\subset \bP^2$ be the branch sextic of the corresponding double cover $\varphi_H\colon S\to \bP^2$. Let $L\subset \bP^2$ be a general line, such that $C\cap L$ is a disjoint union of $6$ points, and denote the divisor corresponding to $C\cap L$ in $\Pic(C)$ by $L$. Then there is a short exact sequence
    \begin{equation}\label{eq: IOOV short exact sequence}
    0\to \frac{\Pic(S)}{\bZ H+2\Pic(S)}\to \left(\frac{\Pic(C)}{\bZ L}\right)_2\overset{\cA}{\to} \Br(S)_2 \to 0. 
    \end{equation}
    Moreover, a Brauer class $\alpha\in \Br(S)_2$ is of Verra type (see Definition \ref{def: verra type}) if and only if there is a line bundle $\cL\in \Jac(C)_2\setminus \{0\}$ such that $\cA(\overline{[\cL]}) = \alpha$.
\end{theorem}

\begin{remark} \label{rem: compatibility of res and A}
    By the purity exact sequence, we have $\Br(\bP^2\setminus C)_2\simeq \Jac(C)_2$. Using this isomorphism, we obtain a commutative diagram
    \[
    \xymatrix{
    \Br(\bP^2\setminus C)_2 \ar[r]^\res \ar[d]_\simeq & \Br(S)_2 \ar[d]^= \\
    \Jac(C)_2 \ar[r]^{\cA|_{\Jac(C)_2}} & \Br(S)_2.
    }
    \]
\end{remark}

The composition of the degree map \eqref{eq: degree map} with the injective map of \eqref{eq: IOOV short exact sequence} induces a group homomorphism 
\begin{equation}\label{eq: degree map for S}
\deg_S\colon \frac{\Pic(S)}{\bZ H+2\Pic(S)}\to\bZ/2\bZ.
\end{equation}

\begin{lemma}\label{lem: degs}
    The degree map $\deg_S$ of \eqref{eq: degree map for S} can be characterised as follows:
    \[
    \function{\deg_S}{\frac{\Pic(S)}{\bZ H+2\Pic(S)}}{\bZ/2\bZ}{\overline{\cO_S(D)}}{D\cdot H \pmod 2}
    \]
\end{lemma}
\begin{proof}
    Firstly, let us explain the definition of the homomorphism
    \begin{equation}\label{eq: IOOV kernel map}
    \frac{\Pic(S)}{\bZ H + 2\Pic(S)} \to \left(\frac{\Pic(C)}{\bZ L}\right)_2,
    \end{equation}
    which was constructed in \cite[Lemma 4.8]{IOOV}. It is defined as the homomorphism induced by the composition
    \[
    \Pic(S) \overset{i^*}{\to} \Pic(R) \overset{\varphi_{H,*}}{\simeq} \Pic(C)
    \]
    where $R$ is the ramification locus of $\varphi_H$, and  $i\colon R\hookrightarrow S$ is the inclusion map. Since $C = 6L \in \Pic(\bP^2)$ and $H = \varphi_H^*L$, it follows that we have $3H = R\in \Pic(S)$. In particular, any curve $C'\subset S$ with $C'\cdot H = n$ restricts to a divisor of degree $3n$ in $\Pic(R)$, and hence also in $\Pic(C)$. The lemma follows from this observation. 
\end{proof}

\begin{proposition}\label{prop: A surjective iff degs surjective}
    The restriction of the map $\cA$ of Theorem \ref{thm: IOOV main theorem} to $\Jac(C)_2$:
    \[
    \cA|_{\Jac(C)_2}\colon \Jac(C)_2 \to \Br(S)_2
    \]
    is surjective if and only if $\dv_{\NS(S)}(H) = 1.$
\end{proposition}
\begin{proof}
    Consider the following diagram whose rows are short exact sequences:
    \begin{equation}
        \xymatrix{
        & \ker(\cA|_{\Jac(C)_2}) \ar[d] & \frac{\Pic(S)}{\bZ H + 2\Pic(S)} \ar[d] \ar[r]^{\deg_S} & \bZ/2\bZ \ar[d] & \\
        0 \ar[r] & \Jac(C)_2 \ar[r] \ar[d]_{\cA|_{\Jac(C)_2}} & \left(\frac{\Pic(C)}{\bZ L}\right)_2 \ar[r] \ar[d]_{\cA} & \bZ/2\bZ \ar[r] \ar[d] & 0 \\
        0 \ar[r] & \Br(S)_2 \ar[r] & \Br(S)_2 \ar[r] & 0 \ar[r] & 0. 
        }
    \end{equation}
    By the Snake Lemma, the map $\cA|_{\Jac(C)_2}$ is surjective if and only if $\deg_S$ is surjective. The result now follows from Lemma \ref{lem: degs}.
\end{proof}

We are now ready to prove Theorem \ref{thm: kernel and cokernel of res}.

\begin{proof}[Proof of Theorem \ref{thm: kernel and cokernel of res}]
    The exact sequence
    \[
    0\to \ker(\res) \to \Br(\bP^2\setminus C)_2 \overset{\res}{\to} \Br(S)_2 \to \operatorname{coker}(\res) \to 0
    \]
    is isomorphic to the sequence
    \[
    0 \to \ker(\cA|_{\Jac(C)_2}) \to \Jac(C)_2 \overset{\cA|_{\Jac(C)_2}}{\to} \Br(S)_2 \to \operatorname{coker}(\cA|_{\Jac(C)_2}) \to 0,
    \]
    see Remark \ref{rem: compatibility of res and A}. By Proposition \ref{prop: A surjective iff degs surjective}, the map $\cA|_{\Jac(C)_2}$ is surjective if and only if $\dv(H) = 1$, and therefore $\res$ is surjective if and only if $\dv(H)=1$. The result now follows form Ford's computation of $\ker(\res)$ \cite{Ford}.
\end{proof}

\begin{definition} \label{def: type I and II}
    Let $S$ be a K3 surface with an ample divisor $H\in \NS(S)$ of degree $H^2 = 2$. Then the polarised K3 surface $(S,H)$ is said to be \textit{of type I} if $\dv(H) = 1$. Otherwise we say that $(S,H)$ is \textit{of type II}.
\end{definition}

\section{Fourier--Mukai partners of twisted K3 surfaces} \label{sec: FM partners of twisted K3 surfaces}
In this section, we carry out the second computation necessary to count Fourier--Mukai partners of Verra fourfolds, namely counting the number of twisted Fourier--Mukai partners of the twisted K3 surface associated to a Verra fourfold. This is relevant because of the following observation.

\begin{remark} \label{rmk: iso verras}
    Suppose $(S,\alpha)$ is a twisted K3 surface of Verra type (see Definition \ref{def: verra type}), and that $(S',\alpha')$ is a twisted K3 surface of Verra type with $\Db(S,\alpha) \simeq \Db(S',\alpha')$ but $(S,\alpha) \not \simeq (S',\alpha')$. Since $(S,\alpha)$ and $(S',\alpha')$ are both of Verra type, there exist Verra fourfolds $V$, $V'$ whose associated twisted K3 surfaces are $(S,\alpha)$ and $(S',\alpha')$, respectively. In other words, $V$ and $V'$ are Fourier--Mukai partners. 

    However, note that any Verra fourfold has two associated twisted K3 surfaces coming from the two projections $V\to \bP^2$. In particular, the Verra fourfolds $V$ and $V'$ above may be isomorphic. This issue will be dealt with later, in Section \ref{sec: FM formula for verra fourfolds}.
\end{remark}

\subsection{Twisted Fourier--Mukai partners}
We let $(S,H)$ be a polarised K3 surface of type I, and let $\alpha\in \Br(S)_2$. Then $\alpha$ is of Verra type for $H$ by Theorem \ref{thm: kernel and cokernel of res}. 
 Following \cite{Ma}, we denote 
    \[
    \FM^2(S,\alpha) \coloneqq \set{(S',\alpha') \mid \Db(S,\alpha) \simeq \Db(S',\alpha') \text{ and } |\alpha'|=2}/_\simeq,
    \]
    and
    \[\cFM^2(S,\alpha) \coloneqq \left\{(S',\alpha')\mid \Db(S,\alpha)\simeq \Db(S',\alpha') \text{ and } |\alpha'|=2\right\}/_\sim\] where 
    \[
    (S',\alpha') \sim (S'',\alpha'') \iff \exists \text{ a Hodge isometry } \varphi \colon T(S') \simeq T(S'') \text{ such that } \varphi^*\alpha'' = \alpha'.
    \]
    Then, by \cite[Lemma 3.2]{Ma}, there is an injective map 
    \[
    \mu\colon \cFM^2(S,\alpha) \hookrightarrow O_{\mathrm{Hodge}}(T(S,\alpha))\backslash I^2(A_{T(S,\alpha)}),
    \]
    where  $I^2(A_{T(S,\alpha)})$ denotes the set of elements $x\in A_{T(S,\alpha)}$ of square $x^2 \equiv 0 \pmod{\bQ/2\bZ}$ and order $2$.

There is a natural quotient map 
\[
\nu\colon \FM^2(S,\alpha) \twoheadrightarrow \cFM^2(S,\alpha).
\]

\begin{remark}\label{rem: split into two problems}
The problem of counting twisted Fourier--Mukai partners of $(S,\alpha)$ now splits into two parts:

\begin{enumerate}
    \item[(1)] Counting the elements of $O_{\mathrm{Hodge}}(T(S,\alpha))\backslash I^2(A_{T(S,\alpha)})$ and particularly the elements in the image of $\mu$;
    \item[(2)] Counting the elements of the fibres of $\nu$.
\end{enumerate}
\end{remark}

For now, we focus on the fibres of the natural quotient map $\nu$.

\begin{lemma} \label{lem: NS(S) of rank 2}
    Let $S$ be a K3 surface of Picard rank $\rho = 2$, and let $H\in \NS(S)$ be an ample divisor of degree $H^2 = 2$. 
    \begin{enumerate}
        \item Assume that $\dv(H) = 1$. Then there is a unique non-negative integer $d\in \bZ$ and a class $D\in \NS(S)$ such that
    \begin{equation} \label{eq: definition of Ld1}
    \NS(S) = \langle H , D \rangle \simeq L_{d,1} \coloneqq \begin{pmatrix}
        2 & 1 \\ 1 & -2d
    \end{pmatrix}.
    \end{equation}
    \item Assume that $\dv(H) = 2$. Then there is a unique non-negative integer $d\in \bZ$ and a class $D\in \NS(S)$ such that
    \begin{equation} \label{eq: definition of Ld2}
    \NS(S) = \langle H , D \rangle \simeq L_{d,2} \coloneqq \begin{pmatrix}
        2 & 2 \\ 2 & -2d
    \end{pmatrix}.
    \end{equation}
    \end{enumerate}
\end{lemma}
\begin{proof}
    Since $H^2 = 2$, it follows that $H$ is a primitive class in $\NS(S)$. Therefore, there exists a class $D' \in \NS(S)$ such that $\NS(S) = \langle H, D'\rangle$. 
    First, assume $\dv(H)=1$. Then $H\cdot D'$ is odd, so we may write $H\cdot D' = 2m+1$. Then $(D'-mH)H = 1$. Hence, writing $D = D-mH$ and $D^2 = -2d$, we have $\NS(S) = \langle H,D\rangle \simeq L_{d,1}$, as required.

    Secondly, assume $\dv(H) = 2$. Then we have $H\cdot D' = 2m$ for some $m\in \bZ$. Now $(D-(m-1)H)H = 2$. Therefore, the basis we are looking for in this case is $\NS(S) = \langle H, D'-(m-1)H\rangle$. 
\end{proof}

\begin{proposition}\label{prop: fibres of nu}
    Let $(S,H,\alpha)$ be a twisted polarised K3 surface of Verra type, and assume that the Picard rank of $S$ is $\rho = 2$.  
    Suppose $(S',\alpha')$ is a twisted K3 surface of Verra type such that $\alpha'$ is non-trivial, and such that $\nu(S,\alpha) = \nu(S',\alpha')$. Then both of the following hold:
    \begin{enumerate}
        \item $S$ and $S'$ are Fourier--Mukai partners;
        \item There is an isometry $\NS(S) \simeq \NS(S')$. 
    \end{enumerate}
\end{proposition}
\begin{proof}
    Since $\nu(S,\alpha) = \nu(S',\alpha')$, we may apply Lemma \ref{lem: overlattice subgroup correspondence}. That is, there is a commutative diagram where the vertical maps are Hodge isometries
    \[
    \xymatrix{
    T(S,\alpha) \ar@{^(->}[r] \ar[d]_\simeq & T(S) \ar[d]^\simeq  \\
     T(S',\alpha') \ar@{^(->}[r] & T(S').
    }
    \]
    It follows that $S$ and $S'$ are Fourier--Mukai partners by the Derived Torelli Theorem for K3 surfaces. 
    Therefore, $\NS(S)$ and $\NS(S')$ are in the same genus.
    
    We assume that $\dv(H)=1$, as the case when $\dv(H)=2$ is completely analogous. In this case, we have $A_{\NS(S')} \simeq \bZ/(4d+1)\bZ$ for some $d\in \bZ_{\geq0}$. Since $(S',\alpha')$ is assumed to be of Verra type, we may apply Lemma \ref{lem: NS(S) of rank 2} to see that $\NS(S') \simeq L_{d,1} \simeq \NS(S)$, as required. 
\end{proof}

\begin{definition}
    Let $(S,\alpha)$ be a twisted K3 surface of Verra type and Picard rank $\rho = 2$. We write
    \[ \FM^2_{\mathrm{Ver}}(S,\alpha) \coloneqq \set{(S',\alpha') \in \nu^{-1}(S,\alpha) \mid (S',\alpha') \text{ is of Verra type}}, \]
    and
    \[
    \FM_{L_{d,i}}(S) \coloneqq \set{S'\in \FM(S) \mid \NS(S') \simeq L_{d,i}}.
    \]    
\end{definition}

\begin{theorem}\label{thm: FML and FMVer}
    Let $(S,H,\alpha)$ be a twisted polarised K3 surface of Verra type and Picard rank $\rho(S)=2$. Assume that $O_{\mathrm{Hodge}}(T(S)) \simeq \bZ/2\bZ$ and that $(S,H)$ is of type I. Then there is a bijection 
    \[
    \FM_{L_{d,1}}(S) 
    \simeq \FM^2_{\mathrm{Ver}}(S,\alpha).
    \]
\end{theorem}
\begin{proof}
    Let us first define the bijection. Fix a Fourier--Mukai partner $S'\in \FM_{L_{d,i}}(S)$ and a Hodge isometry $\psi\colon T(S') \simeq T(S)$.
    Then we let $\alpha' = \psi^*\alpha$. Now we claim that the map 
    \[
    \functionstar{\FM_{L_{d,i}}(S) 
    }{\FM^2_{\mathrm{Ver}}(S,\alpha)}{S'}{(S',\alpha'),}
    \]
    defined by this procedure, is a well-defined bijection.

    There are three things to verify to show that the map is well defined. Once this is done, showing that the map is bijective is straightforward: The injectivity is trivial, since $(S',\alpha') \simeq (S'',\alpha'')$ in particular implies that there is an isomorphism $S'\simeq S''$, and the surjectivity follows from Proposition \ref{prop: fibres of nu}.

    We now verify that the map is well-defined.
    Firstly, we check that the construction of $\alpha'$ is independent of the choice of $\psi$. Since the set of Hodge isometries $T(S')\simeq T(S)$ is a torsor under $O_{\mathrm{Hodge}}(T(S)) \simeq \bZ/2\bZ$, it follows that $\psi$ and $-\psi$ are the only Hodge isometries $T(S')\simeq T(S)$. Since $\alpha\in \Br(S)_2$ is $2$-torsion, we find
    \[
    (-\psi)^*\alpha = (\psi^*\alpha)^{-1} = \psi^*\alpha = \alpha',
    \]
    so that the Brauer class $\alpha'$ is independent of the choice of $\psi$.

    Secondly, we prove that $(S',\alpha')$ is indeed a twisted
    Fourier--Mukai partner of $(S,\alpha)$. 
    
     Firstly, note that, by construction of $\alpha'$, the Hodge isometry $\psi$ restricts to a Hodge isometry $\psi|_{T(S',\alpha')} \colon T(S',\alpha') \simeq T(S,\alpha)$.
     We claim that this can be extended to a Hodge isometry $\tilde H(S',\alpha',\bZ) \simeq \tilde H(S,\alpha, \bZ).$
    
    For this, we apply \cite[Theorem 1.14.2]{Nik80} to the orthogonal complements of $T(S',\alpha')$ and $T(S,\alpha)$ in the respective twisted Mukai lattices: the hypotheses holds since they have rank $4$, signature $(2,2)$ and their discriminant groups have length $2$ (see Lemma \ref{lem: group structure of discriminant} below). 
    Nikulin's theorem guarantees that such orthogonal complements are isometric and that we can choose such an isometry compatibly with the action of $\psi|_{T(S',\alpha')}$ on the discriminant groups, so that the two isometries glue together. As a consequence we get an integral Hodge isometry
    $$\Psi : \tilde H(S',\alpha',\bZ)\simeq \tilde H(S,\alpha,\bZ).$$
    Since $\alpha$ and $\alpha'$ have order $2$, we may assume that $\Psi$ is
    orientation-preserving (see \cite[Remark 4.10]{HStwisted}), so there is an equivalence
    $\Db(S',\alpha')\simeq\Db(S,\alpha)$
    by the twisted derived Torelli theorem, see \cite[Theorem 0.1]{HScaldararu}.
    Moreover, $\alpha'$ is of Verra type by
    Theorem \ref{thm: kernel and cokernel of res}.

    Finally, we show that $\nu(S,\alpha) = \nu(S',\alpha')$. This follows immediately from the construction of $\alpha'$, since $\psi$ induces an isomorphism of short exact sequences
    \[
    \xymatrix{
    0 \ar[r] & T(S',\alpha') \ar[r] \ar@{-->}[d]^\simeq & T(S') \ar[d]^\simeq_\psi \ar[r]^{\alpha'}& \bZ/2\bZ \ar[d]^= \ar[r]& 0 \\
    0 \ar[r] & T(S,\alpha) \ar[r] & T(S)\ar[r]^\alpha & \bZ/2\bZ\ar[r] & 0.
    }
    \]
    Therefore, the map is well-defined and bijective.
\end{proof}

\begin{remark}
The case $(S,H)$ is of type II is more complicated.
To establish a bijection $\FM_{L_{d,2}}(S)\simeq \FM^2_{\mathrm{Ver}}(S,\alpha)$ in this case one would need to prove uniqueness of the embedding $T(S,\alpha) \to H(S,\alpha,\bZ)$
and that $\psi^*\alpha$ is of Verra type
for every $S'\in\FM_{L_{d,2}}(S)$.

Notice that unlike the type I case, the latter condition is not guaranteed by Theorem \ref{thm: kernel and cokernel of res}, since the restriction map $\res$ has cokernel $\bZ/2\bZ$. 
\end{remark}

\subsection{Double planes of type I}

We now move on to the first problem of Remark \ref{rem: split into two problems}.
We first study double planes of type I. Recall from Definition \ref{def: type I and II} that a polarised K3 surface $(S,H)$ of degree 2 is of type I if and only if $\dv(H) = 1$. Equivalently, by Theorem \ref{thm: kernel and cokernel of res}, the restriction map $\res\colon \Br(\bP^2 \setminus C)_2 \to \Br(S)_2$ is surjective, where $C$ is the branch curve of the double cover $\varphi_H\colon S\to \bP^2$. We will restrict our attention to K3 surfaces of Picard rank $\rho = 2$.
Let us first note the following useful property of such K3 surfaces.

\begin{proposition} \label{prop: kernel and cokernel of res for double planes of rank 2}
    Let $(S,H)$ be a polarised K3 surface of degree $H^2 = 2$ of Picard rank $\rho(S) = 2$ and type I. 
    Then every non-trivial Brauer class $\alpha\in \Br(S)_2$ is of Verra type for $H$. Moreover, there is a unique Verra fourfold $V_4$ up to isomorphism whose associated twisted polarised K3 surface is isomorphic to $(S,H,\alpha)$.
\end{proposition}
\begin{proof}
    The map $\res$ of Theorem \ref{thm: kernel and cokernel of res} is an isomorphism. Therefore each non-trivial Brauer class $\alpha \in \Br(S)_2$ is of Verra type. Moreover, $\alpha$ has a unique preimage $\beta \in \Br(\bP^2 \setminus C)_2$, which means there is a unique Verra fourfold $V_4$ with associated K3 surface $(S,H,\alpha)$. 
\end{proof}

\subsubsection{Lattices of double planes of type I}

\begin{lemma} \label{lem: discriminant of Ad}
    The discriminant group $A_{d,1} \coloneqq A_{L_{d,1}}$ is cyclic of order $4d+1$.
\end{lemma}
\begin{proof}
    The order of $A_d$ is equal to the absolute value of the determinant of the matrix \eqref{eq: definition of Ld1}, which is $4d+1$. Moreover, the group structure of $A_d$ can be read off from the Smith Normal Form of \eqref{eq: definition of Ld1}.
\end{proof}

Let $(S,H)$ be a polarised K3 surface of degree $2$, Picard rank $\rho(S) = 2$, and type I, and let $\alpha\in \Br(S)_2$ be a non-trivial Brauer class of order $2$.
Then the twisted transcendental lattice $T(S,\alpha)$ has discriminant $\disc (T(S,\alpha)) = |\alpha|^2\disc (T(S))= 4\disc (T(S))$. Thus, by Lemma \ref{lem: discriminant of Ad}, the group structure of $A_{T(S,\alpha)}$ is given by one of the following two options.
\begin{equation*} \label{eq: group structure of ATSalpha}
A_{T(S,\alpha)} \simeq \begin{cases}
    A_{T(S)} \oplus (\bZ/2\bZ)^{\oplus 2} &\simeq \quad \bZ/(4d+1)\bZ \oplus (\bZ/2\bZ)^{\oplus 2} \\
    A_{T(S)} \oplus \bZ/4\bZ &\simeq \quad \bZ/4(4d+1)\bZ.
\end{cases}
\end{equation*}

We now show that $A_{T(S,\alpha)}$ is never cyclic.

\begin{lemma} \label{lem: group structure of discriminant}
    Let $(S,H)$ be a polarised K3 surface of degree $2$, Picard rank $\rho(S)=2$, and type I, and let $\alpha\in \Br(S)_2$ be a non-trivial Brauer class of order $2$. Then we have an isomorphism of groups
    \[
    A_{T(S,\alpha)} \simeq \bZ/(4d+1)\bZ \oplus (\bZ/2\bZ)^{\oplus 2}.
    \]
\end{lemma}
\begin{proof}
    We mimic the techniques of \cite[\S9]{vGe}.
    Firstly, since $S$ has Picard rank $2$, we have $\NS(S) \simeq L_{d,1}$ for some $d\geq 0$. Denote the standard basis of $L_{d,1}$ by $H,D$, so that $h^2 = 2$, $h\cdot D = 1$ and $D^2 = -2d$. We let $e_1,f_1,e_2,f_2$ be the standard basis for $U^{\oplus 2}$, and consider the primitive embedding
    $\iota \colon L_{d,1}\hookrightarrow U^{\oplus 2}$, given by
    \begin{align*}
        \iota(H) &= e_1 + f_1\\
        \iota(D) &= e_1 + e_2 - df_2.
    \end{align*}
    Then we have
    \[
    N_{d,1}\coloneqq L_{d,1}^\perp = \langle e_1 - f_1 +f_2, e_2 + df_2 \rangle,
    \]
    which has Gram matrix
    \[
    \begin{pmatrix}
    -2 & 1 \\ 1 & 2d    
    \end{pmatrix}.
    \]
    To keep notation light, we will write
    \begin{align*}
        v_1 &\coloneqq e_1 - f_1 + f_2 \\
        v_2 &\coloneqq e_2 + df_2.
    \end{align*}
    Since there is a unique $O(\klat)$-orbit of primitive embeddings $L_{d,1}\hookrightarrow \klat$ by \cite[Theorem 1.14.4]{Nik80}, 
    we may choose a marking $\eta\colon H^2(S,\bZ) \simeq \klat$ which satisfies $\eta(\NS(S)) = \iota(L_{d,1})$, 
    where we consider $\iota(L_{d,1})$ as a sublattice of $\klat$ via the natural inclusion $U^{\oplus 2} \hookrightarrow U^{\oplus 3}\oplus E_8(-1)^{\oplus2} = \klat$.
    This shows that, as abstract lattices, we have an isometry $T(S) \simeq N_{d,1}\oplus U \oplus E_8(-1)^{\oplus2}$. We write 
    \[
    \Lambda \coloneqq U \oplus E_8(-1)^{\oplus 2},
    \]
    so that $T(S) \simeq N_{d,1}\oplus \Lambda.$
    
    The Brauer class $\alpha \in \Br(S)_2 \simeq \Hom(T(S),\bZ/2\bZ)$ can now be written via this identification as
    \[
    \function{\alpha}{N_{d,1}\oplus \Lambda}{\bZ/2\bZ}{mv_1+nv_2+\lambda}{na_\alpha+mb_\alpha + \lambda_\alpha \lambda \pmod 2,}
    \]
    where $a_\alpha,b_\alpha \in \left\{0,1\right\}$ and $\lambda_\alpha \in \Lambda$.
    One can show that the element
    \begin{equation}\label{eq: Balpha}
    B_\alpha \coloneqq a_\alpha v_1 + b_\alpha v_2 + \lambda_\alpha
    \end{equation}
    satisfies
    \begin{equation} \label{eq: alpha is multiplication by an integral B-field mod 2}
    \alpha(x) = B_\alpha \cdot x \pmod 2
    \end{equation}
    for all $x\in T(S)$. 
    We also have the equality $T(S,\alpha) = \ker(\alpha)$ (see (\ref{eq: def of twisted transcendetnal lattice})), thus $\frac{1}{2}B_\alpha\cdot-$ takes integral values on $T(S,\alpha)$ but not on (all of) $T(S)$. In particular, we find a short exact sequence
    \[
    0 \to T(S)^\vee \to T(S,\alpha)^\vee \to \langle \frac{1}{2}\overline{B_\alpha}\rangle \to 0, 
    \]
    where we denote $\frac{1}{2}B_\alpha\cdot -$ by $\frac{1}{2}\overline{B_\alpha}$.
    Since $B_\alpha$ is an integral vector, we see that $\alpha(B_\alpha) = B_\alpha \cdot B_\alpha \equiv 0 \pmod 2$, 
    so that $B_\alpha \in T(S,\alpha)$, hence $\frac{1}{2}\overline{B_\alpha}$ has order $2$ in $A_{T(S,\alpha)}$.
    On the other hand, we also have the short exact sequence
    \[
    0 \to T(S,\alpha) \to T(S) \to \langle \overline{\lambda} \rangle \to 0
    \]
    where $\lambda \in T(S)$ is any element such that $B_\alpha \cdot \lambda$ is odd. Then $\overline{\lambda} \in A_{T(S,\alpha)}$ also has order $2$.
    Finally, the elements
    \[
    \overline{\lambda}, \quad \frac{1}{2}\overline{B_\alpha}
    \]
    are never equal in $A_{T(S,\alpha)}$. This immediately follows from the fact that $\overline{\lambda}\in T(S)/T(S,\alpha)$ but $\frac{1}{2}\overline{B_\alpha}\notin T(S)/T(S,\alpha)$.
    In particular, $A_{T(S,\alpha)}$ contains at least two cyclic subgroups of order $2$, so it is not cyclic.
\end{proof}

It follows from Lemma \ref{lem: group structure of discriminant} that the set $O_{\mathrm{Hodge}}(T(S,\alpha))\backslash I^2(A_{T(S,\alpha)})$ has at most $2$ elements. 

\begin{lemma} \label{lem: isotropic subgroups of twisted discriminant}
    Keeping the assumptions and notation of Lemma \ref{lem: group structure of discriminant}, the number of isotropic subgroups of $A_{T(S,\alpha)}$ is
    \begin{enumerate}
        \item $|I^2(A_{T(S,\alpha)})| = 1$ if $B_\alpha^2 \equiv 2,6 \pmod 8$;
        \item $|I^2(A_{T(S,\alpha)})| = 2$ if $B_\alpha^2 \equiv 0,4 \pmod 8$,
    \end{enumerate}
    where $B_\alpha \in T(S)$ is the class given by \eqref{eq: Balpha}.
\end{lemma}
\begin{proof}
    We already know that $T(S)/T(S,\alpha) \subset A_{T(S,\alpha)}$ is isotropic of order $2$. Let $\zeta \in T(S)/T(S,\alpha)$ be the generator. Then $\zeta = \overline{\lambda}$ where $\lambda\in \Lambda$ satisfies $\lambda\cdot \lambda_\alpha = 1$. Such a $\lambda$ exists because $\lambda_\alpha$ is a primitive element of the unimodular lattice $\Lambda$. 

    Since the $2$-torsion subgroup of $A_{T(S,\alpha)}$ is isomorphic to $(\bZ/2\bZ)^{\oplus 2}$, the three non-zero elements of order $2$ are given by
    \[
        \overline{\lambda}, \quad \frac{1}{2}\overline{B_\alpha}, \quad \frac{1}{2}\left(2\overline{\lambda} + \overline{B_\alpha}\right).
    \]
    We have
    \[
    \left(\frac{1}{2}(2\overline{\lambda}+\overline{B_\alpha})\right)^2 \equiv 1 + \frac{1}{4}B_\alpha^2 \pmod{2\bZ}.
    \]
    Therefore, if $B_\alpha^2 \equiv 0 \pmod 8$, then $\frac{1}{2}\overline{B_\alpha}$ is isotropic, but $\frac{1}{2}(2\overline{\lambda}+B_\alpha)$ is not. 
    If $B_\alpha^2 \equiv 2,6\pmod 8$, then neither $\frac{1}{2}\overline{B_\alpha}$ nor $\frac{1}{2}(2\overline{\lambda}+B_\alpha)$ are isotropic. 
    Finally, if $B_\alpha^2 \equiv 4 \pmod 8$, then $\frac{1}{2}B_\alpha$ is not isotropic, but $\frac{1}{2}(2\overline{\lambda}+B_\alpha)$ is.
\end{proof}

\begin{warning}
    In the notation of Lemma \ref{lem: group structure of discriminant}, the class $\frac{1}{2}B_\alpha$ is a $B$-field lift of $\alpha$. Lemma \ref{lem: isotropic subgroups of twisted discriminant} shows that the number of isotropic elements of $A_{T(S,\alpha)}$ depends only on $B_\alpha^2$. Note, however, that different $B$-field lifts of $\alpha$ potentially have different squares modulo $8$, so one needs to be careful when applying Lemma \ref{lem: isotropic subgroups of twisted discriminant}.
\end{warning}

\begin{remark} \label{rmk: lattice structures in type I}
    In the notation of Lemma \ref{lem: group structure of discriminant}, the lattice structures of $T(S)$ and $T(S,\alpha)$ are given as follows. We already saw in the proof of Lemma \ref{lem: group structure of discriminant} that 
    \[
    T(S) \simeq L_{d,1}(-1) \oplus U \oplus E_8(-1)^{\oplus 2}.
    \]
    The lattice structure of $T(S,\alpha)$ depends on $|I^2(A_{T(S,\alpha)})|$, and one can compute that 
    \begin{equation*}\label{eq: structure of TSalpha}
    T(S,\alpha) \simeq \begin{cases}
        L_{d,1}(-1) \oplus U(2) \oplus E_8(-1)^{\oplus 2} & \text{ if } |I^2(A_{T(S,\alpha)})| = 2 \\
        L_{d,1}(-1) \oplus L_{0,2} \oplus E_8(-1)^{\oplus 2} & \text{ if } |I^2(A_{T(S,\alpha)})| = 1,
    \end{cases}
    \end{equation*}
    where $$L_{0,2} = \begin{pmatrix}
        2 & 2 \\ 2 & 0
    \end{pmatrix}$$
    as in Lemma \ref{lem: NS(S) of rank 2}.
\end{remark}

\begin{lemma}\label{lem: embedding orbits}
    Let $h \in \klat$ be an element of square $h^2 = 2$, and write $\Lambda_2 \coloneqq h^\perp \simeq \langle -2 \rangle \oplus U^{\oplus 2} \oplus E_8(-1)^{\oplus 2}$. Then the number of $O(\Lambda_2)$-orbits of primitive embeddings of $$T_{d,1}\coloneqq L_{d,1}(-1) \oplus U \oplus 
    E_8(-1)^{\oplus 2}$$ into $\Lambda_2$ is equal to
    \[
    |O(A_{T_{d,1}}) / \set{\pm \id}|.
    \]
    In particular, there is a unique $O(\Lambda_2)$-orbit of primitive embeddings $T_{d,t}\hookrightarrow \Lambda_2$ if and only if $4d+1$ is a power of a prime number.
\end{lemma}
\begin{proof}
    Recall from \cite[Proposition 1.5.1]{Nik80} that a primitive embedding $T_{d,t}\hookrightarrow \Lambda_2$ is determined by the data $(N,G,\gamma)$, where $N$ is a lattice of signature $(0,1)$, $G\subset A_{T_{d,t}}$ is a subgroup, and $\gamma\colon G\hookrightarrow A_N$ is
    an injective group homomorphism such that $\Gamma_\gamma^\perp/\Gamma_\gamma \simeq A_{\Lambda_2}$, where $\Gamma_\gamma \subset A_{T_{d,t}}\oplus A_N$ is the 
    graph of $\gamma$. Since $\disc(T_{d,t}) = 4d+1$ is odd, it follows from the equation $\disc(T_{d,t})\cdot \disc(N) = |G^2|\cdot \disc(\Lambda_2) = |G|^2\cdot 2$ that we must have $G = A_{T_{d,t}}$ and $N = \langle 2\cdot(4d+1)\rangle$. In particular, we find $A_N \simeq \gamma(G)\oplus \bZ/2\bZ$. Therefore, there
    are $O(A_{T_{d,t}})$ possible choices for $\gamma$, and two different choices $\gamma_1,\gamma_2$ yield the same $O(\Lambda_2)$-orbit of primitive embeddings $T_{d,t}\hookrightarrow \Lambda_2$ if and only if there is an isometry $\sigma \colon N \simeq N$ such that $\gamma_1 = \sigma^* \circ \gamma_2$, i.e. if and only if $\gamma_1 = \pm \gamma_2$. 
\end{proof}

Finally, we compute $O(L_{d,1})$. This computation will be useful in Section \ref{sec: FM formula for verra fourfolds} for counting Fourier--Mukai partners of Verra fourfolds. As before, we denote the standard basis of $L_{d,1}$ by $H, D$.

To start, we have the obvious two isometries $\pm \id \in O(L_{d,1})$, and we can always find two more isometries $\pm s$, where $s$ is given by 
\[
\isomorphism{s}{L_{d,1}}{L_{d,1}}{aH+bD}{(a+b)H -bD.}
\]
In other words, $s(H) = H$ and $s(D) = H-D$. 

\begin{lemma} \label{lem: few isometries criterion}
    Let $d = k(k+1)$ for some non-negative integer $k$. Then we have $O(L_{d,1}) = \set{\pm \id,\pm s}$.
\end{lemma}
\begin{proof}
    We show that, under the assumption that $d = k(k+1)$, the lattice $L_{d,1}$ contains precisely one vector of square $2$ up to a sign. Indeed, let $H, D\in L_{d,1}$ be the standard basis. In other words, we have $H^2 = 2$, $H\cdot D = 1$, and $D^2 = -2d$. Then there exists a vector $aH + bD \in L_{d,1}$ such that 
    \[
    (aH+bD)^2 = 2a^2 + 2ab -2db^2 = 2
    \]
    and such that $b \neq 0$ if and only if the quantity
    \[
    4b^2 + 16db^2 + 16 = 4\cdot((1+4d)b^2 + 4)
    \]
    is a square. Therefore, such a vector exists if and only if the Pell-type equation
    \begin{equation}\label{eq: pell equation}
    (1+4d)b^2 +4 = c^2
    \end{equation}
    has a solution with $b\neq 0$. Since $d = k(k+1)$, we find that $1+4d = (2k+1)^2$, and it is easy to verify that \eqref{eq: pell equation} has no solutions other than $(b,c) = (0, \pm 2)$ in this case.

    This shows that any isometry $g\colon L_{d,1}\to L_{d,1}$ satisfies $g(H) = \pm H$. Possibly replacing $g$ by $-g$, we may assume that $g(H) = H$. Pick $a,b \in \bZ$ such that $g(D) = aH+bD$, then we must have $2a + b = (aH+bD)\cdot H = g(D)\cdot g(H) = D\cdot H = 1$ and similarly $2a^2 + 2ab - 2db^2 = g(D)^2 = D^2 =  -2d$. It follows from an easy computation that $(a,b) = (0,1),(1,-1)$ are the only two simultaneous solutions to these equations, hence $g = \id,s$.
\end{proof}

\begin{remark}
    In fact, for any two elements $h,h' \in L_{d,1}$ of square $h^2 = (h')^2=2$, there are two isometries $\sigma,\sigma'\in O(L_{d,1})$ with $\sigma(h) = \sigma'(h) = h'$. Since the Pell-type equation \eqref{eq: pell equation} admits infinitely many solutions as soon as $\disc(L_{d,1}) = 1+4d$ is not a square, the group $O(L_{d,1})$ is either infinite or isomorphic to $(\bZ/2\bZ)^{\oplus 2}$.
\end{remark}

\begin{remark}
    The assumption that $d=k(k+1)$ for some non-negative integer $k$ is equivalent to $\disc \ L_{d,1}=1+4d$ being a square. Since a rank $2$ indefinite lattice represents $0$ if and only if its discriminant is a square, it follows that $L_{d,1}$ must contain a non-zero isotropic vector. Indeed, writing $t^2 = 1+4d = (2k+1)^2$ for some $t > 0$, the element 
    \[
    F = \frac{1}{2}(t-1)H + D
    \]
    is a primitive isotropic vector. Moreover, $H,F$ form a basis of $L_{d,1}$ with intersection matrix
    \[
    \begin{pmatrix}
        2 & t \\ t & 0
    \end{pmatrix}.
    \]
    The fact that the orthogonal group of a lattice with such an intersection matrix is isomorphic to $(\bZ/2\bZ)^{\oplus 2}$ was already proved in \cite[Lemma 4.6]{vGe}.
\end{remark}

\subsection{Double planes of type II}
Before we compute the counting formula for Fourier--Mukai partners of a Verra fourfold with an associated K3 surface of type I, we note that there is always at least one Fourier--Mukai partner for type II.

If $(S,H)$ is a polarised K3 surface of Picard rank $\rho(S) = 2$, with degree $H^2 = 2$ and of type II, then the discriminant of $\NS(S)$ is even:

\begin{lemma}
    As a group, $A_{d,2} \coloneqq A_{L_{d,2}}$ is isomorphic to 
    \[
    A_{d,2}\simeq \bZ/2\bZ \oplus \bZ/(2d+2)\bZ.
    \]
\end{lemma}
\begin{proof}
    This follows by computing the Smith normal form of $L_{d,2}$.
\end{proof}

Therefore, $A_{d,2}$ may contain more than $2$ isotropic subgroups of order $2$, potentially making the counting of the elements of $\FM^2(S,\alpha)$ much more complicated. 
We do not provide a counting formula for Fourier--Mukai partners for double planes of type II in the present paper. Nevertheless, we note that there are non-trivial Fourier--Mukai partners in certain cases:

\begin{proposition}
    Let $V_4$ be a Verra fourfold whose associated twisted K3 surface $(S_1,\alpha_1)$ and $(S_2,\alpha_2)$ are double planes of Picard rank $\rho = 2$ and type II. Assume moreover that $\alpha_1$ is non-trivial and that $S_1 \not \simeq S_2$. Then $V_4$ has at least one non-trivial Fourier--Mukai partner.
\end{proposition}
\begin{proof}
    Let $(S_1,\alpha_1)$ be one of the two associated K3 surfaces of $V_4$. Letting $C_1 \subset \bP^2$ be the branch locus of the double cover $S_1\to \bP^2$, it follows from Theorem \ref{thm: birationality theorem} that the map
    \[
    \res\colon \Br(\bP^2 \setminus C_1)_2 \to \Br(S_1)_2
    \]
    satisfies $\ker(\res) \simeq \bZ/2\bZ$. In particular, there exist two Brauer classes $$\beta_1, \beta_2 \in \Br(\bP^2 \setminus C_1)_2$$ such that $\res(\beta_1) = \res(\beta_2) = \alpha_1$.
    These correspond to two non-isomorphic conic bundles $V_{\beta_1}\to \bP^2$ and $V_{\beta_2} \to \bP^2$ 
    and hence induce two non-isomorphic Verra fourfolds, namely $V_4$ and some other Verra fourfold $V'$, which cannot be isomorphic to $V_4$ since $S_1\not \simeq S_2$. Since $V'$ has the same associated twisted K3 surface as $V_4$, it is a Fourier--Mukai partner of $V_4$. 
\end{proof}

\section{Cohomology of Verra fourfolds} \label{sec: cohomology of Verra fourfolds}
In this section, we prove some basic results on the cohomology of Verra fourfolds. We then apply this theory to compare the two twisted K3 surfaces associated with the same Verra fourfold.

\subsection{The basics}
Let $V_4\to \bP^2 \times \bP^2$ be a Verra fourfold, and let $p_1, p_2 \colon V_4 \to \bP^2$ be the projections onto the first, resp. second, factor. Let $H \in \Pic\bP^2$ be an ample generator, and write $h_i \coloneqq p_i^*H \in H^2(V_4,\bZ)$. 

We are mostly interested in $H^4(V_4,\bZ)$ and its relationship with the cohomology of the two associated K3 surfaces of $V_4$. We denote by $A \coloneqq \langle h_1^2,h_1\cap h_2, h_2^2\rangle \subset H^4(V_4,\bZ)$ the ambient cohomology of $V_4$, and write
\[
H^4_\mathrm{prim}(V_4,\bZ) \coloneqq A^\perp \subset H^4(V_4,\bZ)
\]
for the primitive cohomology of $V_4$. 

\begin{lemma}\label{lem: A is integral}
    The sublattice $A\subset H^4(V_4,\bZ)$ is primitive. \footnote{We were informed by Federica Galluzzi and Bert van Geemen that this result will appear in their forthcoming work \cite{GvG26}. We present an independent proof of this result here.}
\end{lemma}
\begin{proof}
    Since the discriminant of $A$ is given by $\disc(A) = 8$, the index of $A\subset A^s$ can be either $1$ or $2$, where $A^s$ is the saturation of $A$ in $H^4(V_4,\bZ)$. If the index equals $[A^s:A] = 2$, then there must be an integral class $x \in A^s\setminus A$ which is given by
    \[
    x = \frac{ah_1^2+b(h_1\cap h_2) + ch_2^2}{2}
    \]
    for some $a,b,c \in \set{0,1}$. However, since the square $x^2 \equiv ac + \frac{b^2}{2}$ must be an integer, it follows that we have $b = 0$. As a consequence, at least one of 
    \begin{equation}\label{eq: these are not integral}
        \frac{1}{2}h_1^2, \qquad \frac{1}{2}h_2^2, \qquad \frac{1}{2}(h_1^2+h_2^2)
    \end{equation}
    is an integral class. We now show that these three classes are all non-integral.
    
    Suppose for a contradiction that $\frac{1}{2}h_1^2$ is integral. Then, under the additional assumption that $V_4$ is symmetric, meaning that the automorphism $s\colon \bP^2 \times \bP^2 \simeq \bP^2 \times \bP^2$ which swaps the two factors extends to an automorphism of $V_4$, the class $\frac{1}{2}h_2^2$ is also integral. However, we have $\frac{1}{2}h_1^2 \cdot \frac{1}{2}h_2^2 = \frac{1}{2}$, so these classes cannot both be integral. As a consequence, $\frac{1}{2}h_1^2$ and $\frac{1}{2}h_2^2$ are both non-integral provided $V_4$ is symmetric. By deforming, it follows that $\frac{1}{2}h_1^2$ and $\frac{1}{2}h_2^2$ are non-integral even for non-symmetric Verra fourfolds. 

    Finally, we rule out the possibility that $\frac{1}{2}(h_1^2+h_2^2)$ is integral. It follows from a standard computation that
    \[
    c_1(V_4) = 2(h_1+h_2), \qquad c_2(V_4) = 2h_1^2 + 7(h_1\cap h_2) + 2 h_2^2.
    \]
    Therefore, by Wu's formula, it follows that the fourth Wu class of $V_4$ is given by $v_4(V_4) \equiv c_2(V_4) \equiv h_1\cap h_2 \pmod 2.$ Thus, if $\frac{1}{2}(h_1^2+h_2^2)$ were integral, it would follow that $$1 =\left(\frac{1}{2}(h_1^2+h_2^2)\right)^2 \equiv \frac{1}{2}(h_1^2+h_2^2) \cdot (h_1\cap h_2) \equiv 0 \pmod 2,$$ a contradiction. Therefore $\frac{1}{2}(h_1^2+h_2^2)$ is also not an integral class.
    Finally, since all three classes in \eqref{eq: these are not integral} are non-integral, it follows that $A$ is a primitive sublattice in $H^4(V_4,\bZ)$. 
\end{proof}

Denote by $(S_1,H_1,\alpha_1)$ and $(S_2,H_2,\alpha_2)$ the two twisted polarised K3 surfaces associated with $V_4$. We denote $H^2_\mathrm{prim}(S_i,\bZ) \coloneqq H_i^\perp \subset H^2(S_i,\bZ)$. 
Following the work of \cite[Th\'eor\`eme II.3.1]{Laszlo} combined with Lemma \ref{lem: A is integral}, it follows that we have short exact sequences:
\begin{equation}\label{eq: laszlos sequence}
0\to H^4_\mathrm{prim}(V_4,\bZ) \overset{\Phi_i}{\to} H^2_\mathrm{prim}(S_i,\bZ)^-(-1) \overset{\gamma_i}{\to} \bZ/2\bZ \to 0 
\end{equation}
for $i = 1,2$.
Here, $H^2_\mathrm{prim}(S_i,\bZ)^-(-1)$ denotes the Tate twist of the Hodge lattice whose underlying group structure is $H^2_\mathrm{prim}(S_i,\bZ)$ and whose bilinear form is the usual one twisted by a sign. The morphism $\Phi_i$ is a Hodge metric morphism, meaning it is a morphism of Hodge structures as well as a lattice embedding. 

We begin our study of the cohomology of Verra fourfolds by computing the underlying lattices. Firstly, we have
\[
A \simeq 
\begin{pmatrix}
    0 & 0 & 2 \\ 0 & 2 & 0 \\ 2 & 0 & 0    
\end{pmatrix} \simeq 
\begin{pmatrix}
    0 & 2 & 0 \\ 2 & 0 & 0 \\ 0 & 0 & 2
\end{pmatrix} = U(2) \oplus \langle 2 \rangle,
\]
where the first matrix is given in the basis $h_1^2, h_1 \cap h_2, h_2^2$ and the second matrix is given in the basis $h_1^2, h_2^2, h_1\cap h_2$. 
Secondly, we can see that $H^4(V_4,\bZ)$ is a unimodular lattice of signature $(21,3)$, which must therefore be isometric to
\[
\Lambda_{V_4} \coloneqq \langle 1 \rangle ^{\oplus 21} \oplus \langle -1 \rangle^{\oplus 3}.
\]
Putting these facts together, we compute the lattice underlying the primitive cohomology $H^4_\mathrm{prim}(V_4,\bZ)$, which is given by the following lattice, which we call the \textit{Verra lattice}
\[
\verlat \coloneqq U(2) \oplus \langle 2 \rangle \oplus U \oplus E_8^{\oplus 2}.
\]

Mimicking the notation for cubic fourfolds, we write $A(V_4) \coloneqq H^{2,2}(V_4) \cap H^4(V_4,\bZ)$ for the algebraic part of $H^4(V_4,\bZ)$, and $T(V_4) \coloneqq A(V_4)^\perp \subset H^4(V_4,\bZ)$ for the \textit{transcendental lattice} of $V_4$. 
The transcendental lattice is a Hodge lattice of K3 type, and it is a primitive sublattice of the primitive cohomology $H^4_\mathrm{prim}(V_4,\bZ)$. For very general Verra fourfolds, the transcendental lattice equals the primitive cohomology.

\begin{lemma}
Let $V_4$ be a smooth Verra fourfold, and let $p_i:V_4\to\bP^2$ be one of its two quadric surface bundle structures. Assume that its discriminant sextic is smooth, and let $(S_i,\alpha_i)$ be the associated twisted K3 surface. Let $Y$ be the associated smooth double EPW quartic. Then there are integral Hodge isometries
$$T(V_4)\simeq T(S_i,\alpha_i)^-(-1)\simeq T(Y)^-.$$
In particular, there is a Hodge isometry $T(S_1,\alpha_1)\simeq T(V_4)^-(-1) \simeq T(S_2,\alpha_2)$.
\end{lemma}
\begin{proof}
By \cite[Theorem II.3.1]{Laszlo}, the Abel--Jacobi map gives an exact
sequence of integral Hodge structures
$$0\rightarrow H^4_\mathrm{prim}(V_4,\bZ)\xrightarrow{\Phi_i}H^2_\mathrm{prim}(S_i,\bZ)^-(-1)\xrightarrow{\gamma_{i}}\bZ/2\bZ\rightarrow0,$$
and satisfies
$(\Phi_i x,\Phi_i y)_S=-(x,y)_{V_4}$.
Thus, after the Tate twist, $\Phi_i$ is an integral Hodge isometry onto
$\ker(\gamma_{i})$. By \cite[3.1]{KKM} and the identification of
\cite[\S9]{vGe}, the restriction of $\gamma_{i}$ to $T(S)$ is $\alpha_i: T(S_i)\longrightarrow\bZ/2\bZ$.
Since $\Phi_i$ is a Hodge isometry, it identifies the transcendental
parts, and therefore
$$T(V_4)\simeq
\ker(\gamma_{i})\cap T(S_i)^-(-1)=\ker(\alpha_i)^-(-1)=T(S_i,\alpha)^-(-1).$$
\end{proof}

\subsection{Comparing the two twisted K3 surfaces}
Let $V_4$ be a Verra fourfold, and let $(S_i,H_i,\alpha_i)$, for $i=1,2$, be the associated polarised twisted K3 surfaces of $V_4$. In this subsection, we compare $(S_1,\alpha_1)$ and $(S_2,\alpha_2)$. The first thing to note is the following:

\begin{lemma} \label{lem: both Verra kuznetsov components equivalent}
    We have $\Db(S_1, \alpha_1) \simeq \Db(S_2, \alpha_2)$. 
\end{lemma}
\begin{proof}
    By \cite[Theorem 1.1]{ADPZ}.
\end{proof}

We begin by investigating the short exact sequence \eqref{eq: laszlos sequence}. We will consider the case $i=1$ for simplicity, as the case $i=2$ simply reverses all the indices. According to \cite[Proposition II.3.2.1]{Laszlo}, the map $\gamma_1 \colon H^2_\mathrm{prim}(S_1,\bZ) \to \bZ/2\bZ$ is given by 
\[
\gamma_1 = (\Phi_1(h_2^2-2x)\cdot-) \pmod 2,
\]
where $x \in H^4(V_4,\bZ)$ is chosen such that $x \cdot h_1^2 = 1$. 

Let us first reinterpret $\gamma_1$ lattice-theoretically. Firstly, an element $x \in H^4(V_4,\bZ)$ with $x\cdot h_1^2 = 1$ can be written as
\[
x = \frac{h_2 + t_1}{2}
\]
for some $t_1 \in H^4_\mathrm{prim}(V_4,\bZ)$. We call the element $t_1 \in H^4(V_4,\bZ)$ the \textit{C\u ald\u araru dual} of $h_2^2$. Using this, we see
\[
\gamma_1 = (\Phi_1(h_2^2-2x)\cdot-) \pmod 2 = (t_1\cdot-)\pmod 2.
\]
Similarly, the C\u ald\u araru dual $t_2$ of $h_1^2$ satisfies $\gamma_2 = (t_2 \cdot-) \pmod 2$. These elements span the lattice $\langle t_1, t_2 \rangle \simeq U(2)$. 

From Lemma \ref{lem: overlattice subgroup correspondence}, the isotropic subgroup of $A_{H^4_\mathrm{prim}(V_4,\bZ)}$ corresponding to the overlattice $\Phi_1\colon H^4_\mathrm{prim}(V_4,\bZ) \hookrightarrow H^2_\mathrm{prim}(S_1,\bZ)^-(-1)$ is given by $\langle \frac{1}{2}\overline{t_2} \rangle\subset A_{H^4_\mathrm{prim}(V_4,\bZ)}$, since $\gamma_1(\frac{1}{2}t_2) = 1$. Vice versa, the subgroup defined by $\Phi_2$ is $\langle \frac{1}{2}\overline{t_1}\rangle$.

\subsubsection{Picard rank 1}

We first apply these observations to the case when $V_4$ is very general, so that $\rho(S_i) = 1$. Then we have $T(V_4) = H^4_\mathrm{prim}(V_4,\bZ)$. 
In particular, the short exact sequence \eqref{eq: laszlos sequence} is simply
\[
0 \to T(V_4) \to T(S_i)^-(-1) \overset{\alpha_i}{\to} \bZ/2\bZ \to 0.
\]
If $O_{\mathrm{Hodge}}(T(V_4)) = \set{\pm \id}$, then the two overlattices $\Phi_i\colon T(V_4) \hookrightarrow T(S_i)^-(-1)$ are not Hodge isometric, because they correspond to two different subgroups of $A_{T(V_4)}$. This proves the following proposition.

\begin{proposition}\label{prop: k3s not isomorphic in general}
    Let $V_4$ be a Verra fourfold with associated twisted K3 surfaces $(S_1,\alpha_1)$ and $(S_2,\alpha_2)$. Assume that $\rho(S_i) = 1$ for $i=1,2$, and assume $V_4$ is very general with this property, so that $O_\mathrm{Hodge}(T(V_4)) \simeq \bZ/2\bZ$. Then we have $(S_1,\alpha_1) \not \simeq (S_2,\alpha_2)$.
\end{proposition}

\subsubsection{Arbitrary Picard rank}

We now move on to higher Picard rank, with the aim of generalising Proposition \ref{prop: k3s not isomorphic in general} to this setting.

\begin{theorem}\label{thm: k3s not isomorphic when both are twisted}
    Let $V_4$ be a Verra fourfold with associated twisted K3 surfaces $(S_1,\alpha_1)$ and $(S_2,\alpha_2)$. Assume the following:
    \begin{itemize}
        \item[(a)] The Brauer classes $\alpha_1 \in \Br(S_1)$ and $\alpha_2 \in \Br(S_2)$ are both non-trivial;
        \item[(b)] We have $O_\mathrm{Hodge}(T(V_4)) \simeq \bZ/2\bZ$. 
    \end{itemize}
    Then we $\nu(S_1,\alpha_1) \neq \nu(S_2,\alpha_2)$. 
\end{theorem}
\begin{proof}
    The short exact sequence \eqref{eq: laszlos sequence} can be restricted to $T(V_4)$, which yields
    \[
    0 \to T(V_4) \to T(S_i)^-(-1) \overset{\alpha_i}{\to} \bZ/2\bZ \to 0.
    \]
    We must compute the isotropic subgroup of $T(V_4)$ defined by this overlattice for $i=1,2$. Recall that $$\gamma_i = (t_1\cdot -) \pmod 2,$$ hence we may identify $H^2_\mathrm{prim}(S_1,\bZ)^-(-1)$ as the subgroup of $H^4_\mathrm{prim}(V_4,\bZ) \otimes \bQ$ given by
    \[
    H^2_\mathrm{prim}(S_1,\bZ)^-(-1) \simeq H^4_\mathrm{prim}(V_4,\bZ) + \frac{1}{2}t_2\cdot \bZ,
    \]
    and $T(S_1)$ is the saturation of $T(V_4)$ in this subgroup:
    \[
    T(S_1)^-(-1) \simeq (T(V_4) \otimes \bQ) \cap H^2_\mathrm{prim}(S_1,\bZ).
    \]
    Since $\alpha_1$ is assumed to be non-trivial, we have $T(S_1)^-(-1) \neq T(V_4)$, hence there is an element $x_1 \in H^4_\mathrm{prim}(V_4,\bZ)$ such that $x_1+\frac{1}{2}t_2 \in T(S_1)^-(-1).$ Then the integral vector $$s_1 \coloneqq 2\cdot(x_1+\frac{1}{2}t_2) = 2x_1+t_2 \in H^4_\mathrm{prim}(V_4,\bZ)$$ is contained in $T(V_4)$, since $T(V_4) = T(S_1)^-(-1) \cap H^4_\mathrm{prim}(V_4,\bZ)$.  
    The overlattice $T(V_4) \hookrightarrow T(S_1)^-(-1)$ corresponds to the isotropic subgroup $\langle\frac{1}{2}\overline{s_1}\rangle \subset A_{T(V_4)}$. 

    Completely analogously, we find an element $x_2 \in H^4_\mathrm{prim}(V_4,\bZ)$ such that the overlattice $T(V_4) \hookrightarrow T(S_2)^-(-1)$ corresponds to the isotropic subgroup $\langle \frac{1}{2}\overline{s_2}\rangle$, where $s_2 = 2x_2 + t_1$.

    Finally, to prove that the two overlattices $T(V_4) \hookrightarrow T(S_i)^-(-1)$ are not Hodge isometric, we must prove that the subgroups $\langle\frac{1}{2}s_i\rangle$ are distinct. Note that $\langle \frac{1}{2}\overline{s_i}\rangle = \langle \frac{1}{2}\overline{s_2} \rangle$ if and only if $\frac{s_1-s_2}{2} \in T(V_4)$, but we have
    \[
    \frac{s_1-s_2}{2} = \frac{t_2-t_1}{2} + x_1-x_2,
    \]
    which is not an integral element of $H^4_\mathrm{prim}(V_4,\bZ)$, so it cannot be contained in $T(V_4)$. Therefore, there is no commutative square of Hodge metric morphisms
    \[
    \xymatrix{
    T(S_1,\alpha_1) \ar[r] \ar[d] & T(S_1) \ar[d] \\
    T(S_2,\alpha_2) \ar[r] & T(S_2).
    }
    \]
\end{proof}

\subsubsection{Picard rank 2}
Finally, we focus our attention on Verra fourfolds whose associated K3 surfaces have Picard rank 2. Recall from Definition \ref{def: type I and II} that we say that a polarised K3 surface $(S,H)$ of degree $2$ is \textit{of type I} if $\dv(H) = 1$, and we say that it is \textit{of type II} if $\dv(H) = 2$. 

\begin{lemma}\label{lem: comparing k3s in rank 2}
    Let $V_4$ be a Verra fourfold with associated twisted polarised K3 surfaces $(S_i,H_i,\alpha_i)$ for $i=1,2$. Suppose $(S_1,H_1)$ is of type I, and let $d \geq 0$ such that $\NS(S_1) \simeq L_{d,1}$ (see Lemma \ref{lem: NS(S) of rank 2}). Assume moreover that $O_{\mathrm{Hodge}}(T(V_4)) \simeq \bZ/2\bZ$. Then exactly one of the following holds:
    \begin{itemize}
        \item[(a)] The surface $(S_2,H_2)$ is also of type I, and $\alpha_1$ and $\alpha_2$ are both non-trivial, and $\nu(S_1,\alpha_1) \neq \nu(S_2,\alpha_2)$:
        \[
        \dv(H_2) = 1, \qquad \alpha_1 \neq 1, \quad \alpha_2 \neq 1, \qquad \nu(S_1,\alpha_1) \neq \nu(S_2,\alpha_2).
        \]
        In this case, we have $\NS(S_2) \simeq L_{d,1} \simeq \NS(S_1)$. 
        \item[(b)] The surface $(S_2,H_2)$ is of type II, and $\alpha_2$ is trivial while $\alpha_1$ is non-trivial. The N\'eron--Severi lattice of $S_2$ is given by $\NS(S_2) \simeq L_{4d,2}$:
        \[
        \dv(H_2) = 2, \qquad \alpha_1 \neq 1, \quad \alpha_2 = 1, \qquad \NS(S_2) \simeq L_{4d,2} = \begin{pmatrix}
        2 & 2 \\ 2 & -8d
    \end{pmatrix}.
        \]
    \end{itemize}
    Moreover, (a) holds if and only if $|I^2(A_{T(S_1,\alpha_1)})| = 2$, and (b) holds if and only if $|I^2(A_{T(S_1,\alpha_1)})| = 1$.
\end{lemma}
\begin{proof}
    Firstly, since $(S_1,H_1)$ is of type I, the Brauer class $\alpha_1$ is non-trivial. This follows from Theorem \ref{thm: kernel and cokernel of res}, since the trivial Brauer class on $S_1$ is not of Verra type. Indeed, the restriction map $\res\colon \Br(\bP^2\setminus C_1)_2 \to \Br(S_1)_2$ is an isomorphism, thus $1 \in \Br(\bP^2\setminus C_1)_2$ is the unique preimage of $1 \in \Br(S_1)_2$. 

    Suppose first that $\alpha_2$ is also non-trivial. Then we have 
    \[
    4\cdot (4d+1) = 4\cdot\disc\NS(S_1) = \disc \ T(S_1,\alpha_1) = \disc \ T(S_2,\alpha_2) =4\cdot \disc\NS(S_2),
    \]
    thus $\NS(S_2)$ has odd discriminant and must be isometric to $L_{d,1}$ by Lemma \ref{lem: NS(S) of rank 2}. It follows that $(S_2,H_2)$ is of type I. Moreover, 
    since $V_4$ satisfies the assumptions of Theorem \ref{thm: k3s not isomorphic when both are twisted}, it follows that $\nu(S_1,\alpha_1) \neq \nu(S_2,\alpha_2)$. 
    Then, by Lemma \ref{lem: isotropic subgroups of twisted discriminant}, we must have $|I^2(A_{T(S_1,\alpha_1)})| = 2$. Indeed, Lemma \ref{lem: isotropic subgroups of twisted discriminant} 
    shows that $I^2(A_{T(S_1,\alpha_1)})$ contains either 1 or 2 elements, and $\nu(S_1,\alpha_1) \neq \nu(S_2,\alpha_2)$ implies $\mu(\nu(S_1,\alpha_1)) \neq \mu(\nu(S_2,\alpha_2)) \in I^2(A_{T(S_1,\alpha_1)})$. 

    Secondly, suppose $\alpha_2$ is trivial. As the trivial Brauer class on a surface of Picard rank 2 and of type I is not of Verra type, $(S_2,H_2)$ must be of type II. Therefore, there is a 
    unique $e \in \bZ$ such that $\NS(S_2) \simeq L_{e,2}$. Computing discriminants, 
    we see that $e = 4d$. With a straightforward computation, one can show that $A_{L_{4d,2}}$ contains precisely $1$ isotropic subgroup of order $2$, so we must also have
    \[
    |I^2(A_{T(S_1,\alpha_1)})| = |I^2(A_{T(S_2)})| = |I^2(A_{\NS(S_2)})| = |I^2(A_{L_{4d,2}})| = 1.
    \]
\end{proof}

Finally, we note the following technical lemma, which will be used in Section \ref{sec: Verra involution}.

\begin{lemma} \label{lem: lattice salads}
    Let $V_4$ be a Verra fourfold with associated twisted polarised K3 surfaces $(S_1,H_1,\alpha_1)$ and $(S_2,H_2,\alpha_2)$, and assume $\NS(S_1) \simeq \NS(S_2) \simeq L_{d,1}$ for some $d\geq 0$ such that $4d+1$ is a power of a prime number. Then there is a commutative diagram of metric morphisms, where the vertical maps are \textit{not necessarily} Hodge metric morphisms:
    \[
    \xymatrix{
    T(S_1,\alpha_1) \ar@{^(->}[r] \ar[d]_\simeq & T(S_1) \ar@{^(->}[r] \ar[d]_\simeq & H_1^\perp \ar[d]_\simeq \\ 
    T(S_2,\alpha_2) \ar@{^(->}[r] & T(S_2) \ar@{^(->}[r] & H_2^\perp.
    }
    \]
\end{lemma}
\begin{proof}
    There is a commutative square
    \[
     \xymatrix{
    T(S_1,\alpha_1) \ar@{^(->}[r] \ar[d]_\simeq & T(S_1) \ar[d]^\Psi_\simeq\\ 
    T(S_2,\alpha_2) \ar@{^(->}[r] & T(S_2) 
    }
    \]
    Since $4d+1$ is a power of a prime, there is a unique $O(H_1^\perp)$-orbit of embeddings $T(S_1)\hookrightarrow H_1^\perp$ by Lemma \ref{lem: embedding orbits}. Therefore $\Psi$ extends to an isometry $\tilde\Psi\colon H_2^\perp \simeq H_1^\perp$, as required.
\end{proof}

\section{The Fourier--Mukai number of a Verra fourfold} \label{sec: FM formula for verra fourfolds}
The main goal of this section is to prove Theorem \ref{thm: counting formula}, which is a counting formula for Fourier--Mukai partners of Verra fourfolds whose associated K3 surfaces are double planes of Picard rank $2$ and type I. We first note the following general fact:

\begin{theorem} \label{thm: finitely many FM partners}
    Let $V_4$ be a Verra fourfold. Then $|\FM(V_4)|$ is finite. That is, $V_4$ has finitely many Fourier--Mukai partners. 
\end{theorem}
\begin{proof}
    We denote by $(S,\alpha)$ one of the two associated twisted K3 surfaces of $V_4$.
    For any Fourier--Mukai partner $V\in \FM(V_4)$, the two associated K3 surfaces of $V$ are each Fourier--Mukai partners of the associated twisted K3 surfaces of $V_4$. There are finitely many of those by \cite[Corollary 0.5]{HStwisted}. Thus, for any $(T,\beta) \in \FM(S,\alpha)$, we must prove that there are finitely many Verra fourfolds for which one of the associated twisted K3 surfaces is $(T,\beta)$. This follows from combining Lemma \ref{lem: finitely many double covers} with \cite[Theorem 2.4]{Ford} and \cite[Lemma 5.9]{JM}, since the kernel of the map
    \[
    \res \colon \Br(\bP^2\setminus C)_2 \to \Br(T)_2
    \]
    is finite.
\end{proof}

\begin{lemma}\label{lem: finitely many double covers}
    Let $S$ be a K3 surface. Then $S$ admits finitely many double covers $S\to \bP^2$ up to isomorphism, where we say that two double covers $f_1\colon S \to \bP^2$, $f_2 \colon S \to \bP^2$ are isomorphic if there exist automorphisms $\varphi \colon S \simeq S$, $\overline{\varphi}\colon \bP^2 \simeq \bP^2$ such that $\overline{\varphi}\circ f_1=f_2\circ \varphi$.
\end{lemma}
\begin{proof}
    This is a direct application of \cite[Theorem 0.1(b)]{Sterk} 
\end{proof}

\begin{definition}
    For a K3 surface $S$, we write $$\mathrm{Amp}_2(S)\coloneqq \set{H \in \NS(S) \mid H \text{ is ample and } H^2=2}\big/{\Aut(S)}$$
    for the set of isomorphism classes of double covers $S\to \bP^2$. This is a finite set by Lemma \ref{lem: finitely many double covers}.
\end{definition}

\begin{example}
    In the following three cases, it is easy to compute the size of $\mathrm{Amp}_2(S)$.
    \begin{enumerate}
        \item If $\rho(S)=1$, then $\mathrm{Amp}_2(S)$ is either empty, or a singleton, depending on the square of the ample generator of $\NS(S)$.
        \item If $\NS(S) \simeq L_{d,1}$ with $d= k(k+1)$ for some $k\geq 0$, we claim that $\mathrm{Amp}_2(S)$ is a singleton. Indeed, we proved this in Lemma \ref{lem: few isometries criterion}. In this case, $\disc(\NS(S))$ is a square, hence $S$ admits an elliptic fibration (see e.g. \cite{vGe,MS24}). In fact, a K3 surface $S$ with $\NS(S) \simeq L_{d,1}$ admits an elliptic fibration if and only if $d = k(k+1)$ for some integer $k \geq 0$.
        \item If $\NS(S) \simeq L_{d,1}$ with $\disc(\NS(S)) = 4d+1$ a power of a prime number, then $\mathrm{Amp}_2(S)$ is a singleton. We show this in Lemma \ref{lem: power of a prime? one double cover for you}
    \end{enumerate}
\end{example}

\begin{lemma}\label{lem: power of a prime? one double cover for you}
    Let $S$ be a K3 surface with $\NS(S) \simeq L_{d,1}$, where $4d+1$ is a power of a prime number. Then $|\mathrm{Amp}_2(S)| = 1$.
\end{lemma}
\begin{proof}
    Let $H_1,H_2 \in \NS(S)$ be two ample divisors of degree $2$. Then there exists a isometry $\varphi\in O(\NS(S))$ such that $\varphi(H_1) = H_2$. Moreover, since $O(A_{\NS(S)}) = \set{\pm \id}$, we can extend $\varphi$ to an effective Hodge isometry of $H^2(S,\bZ)$, hence there is an isomorphism $f\colon S \simeq S$ with $f^*|_{\NS(S)} = \varphi$, and therefore $[H_1] = [H_2] \in \mathrm{Amp}_2(S)$. 
\end{proof}

\subsection{Double planes of type I}

\begin{remark}
    For the proof of the Counting Formula in Theorem \ref{thm: counting formula} below, we need to assume $O_{\mathrm{Hodge}}(T(S_1,\alpha_1))\simeq \bZ/2\bZ$ to apply Lemma \ref{lem: comparing k3s in rank 2}, but we also need $O_\mathrm{Hodge}(T(S_1)) \simeq \bZ/2\bZ$ to apply Theorem \ref{thm: FML and FMVer}. We note that both of these isomorphisms hold under the assumption $O_{\mathrm{Hodge}}(T(S_1)_\bQ)\simeq \bZ/2\bZ$, which holds for very general K3 surfaces in any Noether--Lefschetz locus for Picard rank $\rho  \leq 19$, see for example \cite[Lemma 3.1]{MeinsmaOnLEquivalence}. 
\end{remark}

\begin{theorem}\label{thm: counting formula}
    Let $V_4$ be a Verra fourfold, and let $(S_1,H_1,\alpha_1)$ and $(S_2,H_2,\alpha_2)$ be the two associated twisted polarised K3 surfaces of $V_4$. Assume $\NS(S_1) \simeq L_{d,1}$ for some $d\in \bZ$, and assume that $O_{\mathrm{Hodge}}(T(S_1)_\bQ)\simeq \bZ/2\bZ$. 
    Then we have
    \begin{align*}
    |\FM(V_4)| &= |\mathrm{Amp}_2(S_1)|\cdot |\FM_{L_{d,1}}(S_1)| \\
    &= |\mathrm{Amp}_2(S_1)|\cdot |O(\NS(L_{d,1})) \backslash O(A_{L_{d,1}})/O_{\mathrm{Hodge}}(T(S_1))|. 
    \end{align*}
\end{theorem}
\begin{proof}
    There are two possibilities for $(S_2,H_2,\alpha_2)$, which we consider separately.

    \textit{Case 1. $\alpha_2$ is non-trivial.}
    
    Suppose $V \in \FM(V_4)$. Let $(T_1,L_1,\beta_1)$ and $(T_2,L_2,\beta_2)$ be the two associated twisted K3 surfaces of $V$. Then $(T_1,\beta_1), (T_2,\beta_2) \in \FM(S,\alpha)$ since $V$ is a Fourier--Mukai partner of $V_4$. 
    
    Note that $V_4$ satisfies the conclusions of Lemma \ref{lem: comparing k3s in rank 2}(a), since $\alpha_1$ and $\alpha_2$ are both non-trivial. In particular, we find 
    \[
    |I^2(A_{T(S_1,\alpha_1)})| = 2,
    \]
    and since $T(T_i,\beta_i) \simeq T(S_i,\alpha_i)$, case (a) of Lemma \ref{lem: comparing k3s in rank 2} also applies to $V$. In particular, we have $\nu(T_1,\beta_1) \neq \nu(T_2,\beta_2)$. Assume without loss of generality that $\nu(T_1,\beta_1) = (S_1,\alpha_1)$. Then we have $(T_1,\beta_1) \in \FM_{\mathrm{Ver}}^2(S_1,\alpha_1)$.
    We claim that this constructs a well-defined surjection
    \[
    \isomorphism{\varphi}{\FM(V_4)}{\FM_{\mathrm{Ver}}^2(S_1,\alpha_1)}{V}{(T_1,\beta_1).}
    \]
    The fact that it is well-defined follows from the construction, since the construction of $(T_1,\beta_1)$ from $V$ is uniquely determined by the above procedure. The map is also surjective, by definition of $\FM_{\mathrm{Ver}}^2(S_1,\alpha_1)$.

    We claim that the preimage of any $(T,\beta)\in \FM^2_\mathrm{Ver}(S_1,\alpha_1)$ along $\varphi$ contains precisely $|\mathrm{Amp}_2(S_1)|$ elements. Indeed, once we fix an ample divisor $H \in \NS(T)$ of degree 2, we may apply Theorem \ref{thm: kernel and cokernel of res} to show that
    \[
    \res\colon \Br(\bP^2 \setminus C)_2 \simeq \Br(T)_2
    \]
    is an isomorphism. In particular, the triple $(T,H,\beta)$ determines a unique Verra threefold $V_3 \to \bP^2$, and hence a unique Verra fourfold by \cite[Proposition 4.23]{JM}. In other words, the number of elements in the preimage of $(T,\beta)$ is $|\mathrm{Amp}_2(T)| = |\mathrm{Amp}_2(S)|$, as claimed.

    Now it follows from Theorem \ref{thm: FML and FMVer} that there is a bijection
    \[
    \FM(V_4) \simeq \FM_{L_{d,1}}(S),
    \]
    and the size of $\FM_{L_{d,1}}(S)$ is computed in \cite[Theorem 2.3]{HLOY}.

    \textit{Case 2. $\alpha_2$ is trivial.}

    The proof in this case is very similar to the proof in Case 1, except now $V_4$ and all of its Fourier--Mukai partners satisfy the conclusions of Lemma \ref{lem: comparing k3s in rank 2}(b). To be precise, suppose $V \in \FM(V_4)$, and let $(T_1,\beta_1)$ and $(T_2,\beta_2)$ be the two associated twisted K3 surfaces of $V$. Since $|I^2(A_{T(S_1,\alpha_1)})| = |I^2(A_{T(T_1,\beta_1)})| =1$, $V$ satisfies Lemma \ref{lem: comparing k3s in rank 2}(b). In particular, exactly one of $\beta_1$ and $\beta_2$ is trivial. Assume without loss of generality that $\beta_1 \neq 1 \in \Br(T_1)$, then we claim that the assignment
    \[
    \isomorphism{\varphi}{\FM(V_4)}{\FM^2_{\mathrm{Ver}}(S_1,\alpha_1)}{V}{(T_1,\beta_1)}
    \]
    is a well-defined surjection, all of whose fibres have size $|\mathrm{Amp}_2(S)|$. We omit the proof of this fact, as the argument is completely analogous to Case 1. 
\end{proof}

In Theorem \ref{thm: counting formula}, the group $O(A_{L_{d,1}})$ is easy to compute. Namely, as $A_{L_{d,1}}$ is cyclic of order $4d+1$, we find that
\[
O(A_{L_{d,1}}) \simeq (\bZ/2\bZ)^{\omega(4d+1)},
\]
where $\omega(4d+1)$ is the number of prime divisors of $4d+1$.
Using this, we obtain a categorical Torelli theorem for very general Verra fourfolds contained in certain Noether--Lefschetz loci. On the other hand, we can show that the number of Fourier--Mukai partners of a Verra fourfold is unbounded.

\begin{corollary} \label{cor: categorical torelli for rank 2}
    Let $V_4$ be a Verra fourfold whose associated twisted K3 surfaces satisfy $\NS(S) \simeq \NS(S') \simeq L_{d,1}$ such that $4d+1$ is a power of a prime number. Then for any Verra fourfold $V_4'$, we have
    \[
    \cA_{V_4} \simeq \cA_{V_4'} \iff V_4 \simeq V_4'.
    \]
\end{corollary}
\begin{proof}
    This follows immediately from Theorem \ref{thm: counting formula} combined with Lemma \ref{lem: power of a prime? one double cover for you}, since we have $O(A_{L_{d,1}}) \simeq \bZ/2\bZ$.
\end{proof}

\begin{corollary}\label{cor: unbounded}
    Let $N$ be a positive integer. Then there are $N$ Verra fourfolds $V_1,\cdots, V_N$, all of whose associated twisted K3 surfaces have Picard number $2$, such that
    for all $1\leq i\neq  j \leq N$, we have $\cA_{V_i}\simeq \cA_{V_j}$ and $V_i \not\simeq V_j$.
\end{corollary}
\begin{proof}
    Let $S$ be a K3 surface with $\NS(S) \simeq L_{d,1}$, and suppose we have $O_{\mathrm{Hodge}}(T(S)) = \set{\pm \id}$.
    Then we have
    \[
    O(A_{\NS(S)}) \simeq (\bZ/2\bZ)^{\omega(4d+1)},
    \]
    
    where $\omega(4d+1)$ is the number of prime divisors of $4d+1$. 
    Let $\alpha$ be a Brauer class on $S$ of Verra type, and let $V$ be the unique Verra fourfold whose associated K3 surface is $(S,\alpha)$.
    
    By Lemma \ref{lem: few isometries criterion}, if $4d+1$ is a square with at least $\log_2(N) + 2$ prime divisors, we compute
    \[
    |\FM(V)| =|O(L_{d,1})\backslash O(A_{L_{d,1}})/O_\mathrm{Hodge}(T(S))| \geq 2^{\log_2(N)+2-2} = N,
    \]
    as required.
\end{proof}

\subsection{Fourier--Mukai partners of Verra threefolds}
In this subsection, we relate the Fourier--Mukai number of a Verra threefold to the Fourier--Mukai number of the associated Verra fourfolds, computed in the previous subsection. We first observe a relationship between the equivariant Kuznetsov component of the Verra fourfold, and the Kuznetsov component of the Verra threefold. For a recollection of equivariant categories, see for example \cite[Section 3]{KP}.

\begin{proposition} \label{prop: Ku V_4 V_3 equivalence}
    Let $V_4$ be a Verra fourfold with branch locus $V_3$, and assume that $V_3$ is a smooth ordinary Verra threefold. Then we have
    \[
    \cA_{V_4}^{\mu_2}\simeq \cA_{V_3}.
    \]
\end{proposition}

\begin{proof}
    Note that the derived category of $\bP^2 \times \bP^2$, the base of the double cover $V_4$, admits a rectangular Lefschetz decomposition. Hence, \cite[Theorem 1.1]{KP} gives the desired result.
\end{proof}

\begin{lemma} \label{lem: equivalence descent}
    Let $V_4$ and $V_4'$ be Verra fourfolds whose respective branch loci are smooth ordinary Verra threefolds $V_3$ and $V_3'$. Then
    \[
    \cA_{V_3}\simeq \cA_{V_3'} \implies \cA_{V_4}\simeq \cA_{V_4'} .
    \]
\end{lemma}

\begin{proof}
    By Proposition \ref{prop: Ku V_4 V_3 equivalence}, we have $\cA_{V_4}^{\mu_2}\simeq \cA_{V_3}$. Denote by $\tau$ the generator of the $\widehat{\mu}_2$-action on $\cA_{V_3}$, where $\widehat{\mu}_2$ is the dual group of $\mu_2$. Indeed, $\cA_{V_3}$ is an Enriques category and $\tau$ is the associated involutive functor of $\cA_{V_3}$. In particular, the Serre functor of $\cA_{V_3}$ is given by $S_{\cA_{V_3}} \simeq \tau[2]$.

    Since the Serre functors of $\cA_{V_3}$ and $\cA_{V_3'}$ are the respective involutive functors generating the action of $\widehat{\mu}_2$ on $\cA_{V_3}$ and $\cA_{V_3'}$, an equivalence $\cA_{V_3} \simeq \cA_{V_3'}$ descends to an equivalence $ \cA_{V_3}^{\widehat{\mu}_2} \simeq  \cA_{V_3'}^{\widehat{\mu}_2}$, by the same argument as in \cite[Theorem 9.9]{JLLZ}, \cite[Lemma 6.2]{DJR}, \cite[Theorem 6.1]{JM}. By \cite[Theorem 1.3]{ElaginTri} and Proposition \ref{prop: Ku V_4 V_3 equivalence}, we also have $\cA_{V_4} \simeq \cA_{V_3}^{\widehat{\mu}_2}$. Hence, we get an equivalence $\cA_{V_4} \simeq \cA_{V_4'}$.
\end{proof}

\begin{corollary} \label{cor: bound on FM partners of verra threefolds}
    Let $V_4$ be a Verra fourfold with branch locus $V_3$. Then we have 
    \[
    |\FM(V_3)| \leq |\FM(V_4)|.
    \]
\end{corollary}

We note that non-trivial Fourier--Mukai partners of Verra threefolds have not been constructed yet. Corollary \ref{cor: bound on FM partners of verra threefolds}, combined with Theorem \ref{thm: finitely many FM partners} shows that any Verra threefold has finitely many Fourier--Mukai partners.
We propose the following question in search of Fourier--Mukai partners of Verra threefolds.

\begin{question}
    Suppose we have a polarised K3 surface $(S,H)$ of degree $H^2 = 2$, and suppose that $\beta,\beta'\in \Br(\bP^2\setminus C)_2$ have $\res(\beta) = \res(\beta')$, where $C\subset \bP^2$ is the branch curve of the double cover $\varphi_H \colon S \to \bP^2$. Then are $V_\beta$ and $V_{\beta'}$ Fourier--Mukai partners?
\end{question}

\section{A derived Torelli theorem for Verra fourfolds}

The Kuznetsov component of a Verra fourfold is equivalent to the derived category of a twisted K3 surface. Therefore, the derived Torelli theorem for twisted K3 surfaces \cite[Theorem 0.1]{HScaldararu} induces a derived Torelli theorem for Verra fourfolds. Namely, two Verra fourfolds are Fourier--Mukai partners if and only if there is an orientation-preserving Hodge isometry between the twisted Mukai lattices of their K3 surfaces.

In this section, we show that the existence of a Hodge isometry of transcendental lattices is equivalent to the existence of a Fourier--Mukai equivalence of Kuznetsov components for large classes of Verra fourfolds. Thereby, we establish another derived Torelli theorem for these Verra fourfolds, which is used in Subsection \ref{subsec: Lequivalence} to study $\bL$-equivalence for Verra fourfolds.

\begin{theorem}\label{thm: derived torelli for verra fourfolds}
    Let $V_4$ be a Verra fourfold with associated twisted K3 surfaces $(S_1,\alpha_1)$ and $(S_2,\alpha_2)$. Assume $\rho(S_i) = 2$ for $i=1,2$, and suppose $S_1$ is of type I. Let $V_4'$ be a Verra fourfold. Then the following are equivalent:
    \begin{enumerate}
        \item There is a Fourier--Mukai equivalence $\cA_{V_4} \simeq \cA_{V_4'}$;
        \item There is a Hodge isometry $T(V_4) \simeq T(V_4')$.
    \end{enumerate}
\end{theorem}
\begin{proof}
    Let $(S_1',\alpha_1')$ and $(S_2', \alpha_2')$ be the twisted K3 surfaces associated with $V_4'$. 

    If $V_4$ and $V_4'$ are Fourier--Mukai partners, then we have $\Db(S_1,\alpha) \simeq \cA_{V_4}\simeq \cA_{V_4'}\simeq \Db(S_1',\alpha_1')$ and in particular there is a Hodge isometry \[T(V_4) \simeq T(S_1,\alpha_1)^-(-1) \simeq T(S_1',\alpha_1')^-(-1) \simeq T(V_4'). \] 

    Conversely, suppose that we have a Hodge isometry $T(V_4) \simeq T(V_4')$. As before, this induces a Hodge isometry $T(S_1,\alpha_1) \simeq T(S_1',\alpha_1')$. 
    Denote by $T$ the underlying lattice structure of $T(S_1)$. Note that $T$ is also the underlying lattice of $T(S_1')$, since $T(S_1,\alpha_1)$ only has one overlattice of index $2$ up to isometry. 
    By Lemma \ref{lem: isotropic subgroups of twisted discriminant} and Lemma \ref{lem: overlattice subgroup correspondence}, there are at most $2$ isomorphism classes of lattice embeddings $T(S_1,\alpha_1) \hookrightarrow T$, and we obtain two such embeddings. The first one is given by the natural embedding $T(S_1,\alpha_1) \hookrightarrow T(S_1)$, and the second embedding is the composition $T(S_1,\alpha_1) \simeq T(S_1',\alpha_1') \hookrightarrow T(S')$. 

    Suppose first that these two embeddings are isomorphic. That is, assume that there is a commutative diagram of Hodge lattice embeddings:
    \[
    \xymatrix{
    T(S_1,\alpha_1) \ar@{^(->}[r] \ar[d]_\simeq & T(S_1) \ar[d]_\simeq \\
    T(S_1',\alpha_1') \ar@{^(->}[r] & T(S_1').
    }
    \]
    Then, as in the proof of Theorem \ref{thm: FML and FMVer}, there is an equivalence $\Db(S) \simeq \Db(S')$ which induces an equivalence $\Db(S,\alpha) \simeq \Db(S',\alpha')$. Thus, in this case $V_4$ and $V_4'$ are Fourier--Mukai partners.

    On the other hand, if the two embeddings are not isomorphic, then there are two elements in $O_\mathrm{Hodge}(T(S_1,\alpha_1)\backslash I^2(A_{T(S_1,\alpha_1)})$, and we have $\nu(S_1',\alpha_1') \neq \nu(S_2',\alpha_2')$. Therefore, we have $\nu(S_1',\alpha_1') = \nu(S_2,\alpha_2)$. Now it follows from the steps above that there is an equivalence $\Db(S_1',\alpha_1') \simeq \Db(S_2,\alpha_2)$, hence $V_4$ and $V_4'$ are Fourier--Mukai partners in this case as well. 
\end{proof}

We note the following Torelli theorem for very general Verra fourfolds, which can also be thought of as a derived Torelli theorem.

\begin{proposition}\label{prop: torelli for very general verra with transcendental}
    Let $V_4$ be a Verra fourfold whose associated twisted K3 surface $(S,\alpha)$ has Picard rank $\rho(S)=1$, and whose associated K3 surfaces are not isomorphic. Let $V_4'$ be a Verra fourfold such that there is a Hodge isometry $T(V_4) \simeq T(V_4')$. Then $V_4$ and $V_4'$ are isomorphic.
\end{proposition}
\begin{proof}
    Choose an associated twisted K3 surface $(S',\alpha')$ of $V_4'$ such that there is a commutative square
    \[
    \xymatrix{
    T(S,\alpha) \ar@{^(->}[r] \ar[d]_\simeq & T(S) \ar[d]_\simeq \\
    T(S',\alpha') \ar@{^(->}[r] & T(S').
    }
    \]
    Such a twisted K3 surface exists because $T(S,\alpha)$ has at most two overlattices of index $2$, up to Hodge isometries \cite[\S9.8]{vGe}. In particular, $S$ and $S'$ are Fourier--Mukai partners by the derived Torelli theorem for K3 surfaces. Thus we have an isomorphism $f\colon S\simeq S'$ with $f^*\alpha' = \alpha$, since $S$ has no non-trivial Fourier--Mukai partners \cite[Corollary 2.2]{StellariFMrankonetwo}. Therefore, $V_4$ and $V_4'$ have isomorphic associated twisted K3 surfaces, thus they are isomorphic by the categorical Torelli theorem for very general Verra fourfolds \cite[Theorem 7.2]{JM}.
\end{proof}

Note that, in the setting of Proposition \ref{prop: torelli for very general verra with transcendental}, we have $T(V_4) = H^4_\mathrm{prim}(V_4,\bZ)$. Therefore, Proposition \ref{prop: torelli for very general verra with transcendental} asserts that two very general Verra fourfolds are isomorphic if and only if their primitive cohomologies are Hodge isometric. 

The results in this section raise two important questions:

\begin{question}
\leavevmode
    \begin{itemize} 
        \item[(1)] \normalfont{(Derived Torelli for Verra fourfolds)}
        Which Verra fourfolds satisfy derived Torelli? That is, can we classify the Verra fourfolds $V_4$ for which the existence of an equivalence $\cA_{V_4}\simeq \cA_{V_4'}$ is equivalent to the existence of a Hodge isometry $T(V_4) \simeq T(V_4')$?
        \item[(2)] \normalfont{(Torelli for Verra fourfolds)} Let $V_4$ and $V_4'$ be Verra fourfolds. Is it true that $V_4$ and $V_4'$ are isomorphic if and only if there is a Hodge isometry $H^4_\mathrm{prim}(V_4,\bZ) \simeq H^4_\mathrm{prim}(V_4',\bZ)$?
    \end{itemize}
\end{question}

\subsection{$\bL$-equivalence for Verra fourfolds} \label{subsec: Lequivalence}
We now apply Theorem \ref{thm: derived torelli for verra fourfolds} to $\bL$-equivalence. Our main references for $\bL$-equivalence are \cite{KS,MeinsmaOnLEquivalence,Mei}. 

Two varieties $X,Y$ are \textit{$\bL$-equivalent}, denoted $X\sim_\bL Y$, if there is a non-negative integer $n \in \bZ$ such that 
\[
\bL^n \cdot ([X]-[Y]) = 0 \in \gravc,
\]
where $\gravc$ is the Grothendieck ring of varieties, and $\bL = [\bA^1]$ is the class of the affine line. It is conjectured that $\bL$-equivalence implies derived equivalence for hyperk\"ahler manifolds, and these ideas extend to Fano varieties of K3 type \cite{MM,BL}:

\begin{conjecture}[Fano $\bL$-equivalence conjecture] \label{conj: Fano L equivalence conjecture}
    Let $X$, $Y$ be Fano varieties of K3 type, with Kuznetsov components $\cA_X$ and $\cA_Y$, and assume $X$ and $Y$ are $\bL$-equivalent. Then there is an equivalence $\cA_{X} \simeq \cA_{Y}$.
\end{conjecture}

The main strategy for tackling Conjecture \ref{conj: Fano L equivalence conjecture} is by making use of the Hodge realisation map $\gravc \to \gghs$ combined with derived Torelli theorems, first carried out by Efimov \cite{Efi18}. 

\begin{lemma}\label{lem: L equivalence implies Hodge isometry}
    Let $V_4$ and $V_4'$ be Verra fourfolds, and assume that we have $V_4\sim_\bL V_4'$. Assume moreover that $\End_{\mathrm{Hodge}}(T(V_4)) \simeq \bZ$ and that $\rk(T(V_4)) \neq 4$. Then there is a Hodge isometry $T(V_4) \simeq T(V_4')$. 
\end{lemma}
\begin{proof}
    Recall from \cite{MeinsmaOnLEquivalence} that we have the discriminant map:
    \[
    \functionstar{K_0(\operatorname{Var}_\bC)}{\bQ^\times}{\text{[}X]}{|\left(\frac{H^4(X,\bZ)}{T(X) \oplus H^{2,2}(X,\bZ)}\right)_{\operatorname{tors}}|}
    \]
    which takes value $\disc(T(V_4))$ on $[V_4]$. Therefore, we can apply \cite[Proposition 3.11]{MeinsmaOnLEquivalence} to obtain a Hodge isometry $T(V_4) \simeq T(V_4')$.
\end{proof}

\begin{theorem}\label{thm: L implies D for Verra of type I}
     Let $V_4$ be a Verra fourfold, and let $(S,H,\alpha)$ be one of the two associated twisted polarised K3 surfaces of $V_4$. Assume that $V_4$ satisfies one of the following:
     \begin{itemize}
         \item[(a)] $V_4$ is very general, that is, we have $\End_\mathrm{Hodge}(T(V_4)) \simeq \bZ$ and $\rho(S) = 1$;
         \item[(b)] We have $\rho(S) = 2$, and $(S,H)$ is a double plane of type I, and $V_4$ is very general with these properties, that is, we have $\End_\mathrm{Hodge}(T(V_4)) \simeq \bZ$.
     \end{itemize}
     Then, for any Verra fourfold $V_4'$, if $V_4$ and $V_4'$ are $\bL$-equivalent, then $V_4$ and $V_4'$ are Fourier--Mukai partners:
     \[
     V_4 \sim_\bL V_4' \implies \cA_{V_4}\simeq \cA_{V_4'}.
     \]
\end{theorem}
\begin{proof}
    By Lemma \ref{lem: L equivalence implies Hodge isometry}, $\bL$-equivalent Verra fourfolds have Hodge isometric transcendental lattices. The result follows from this fact, combined with the derived Torelli theorems for Verra fourfolds satisfying \textit{(a)} (Proposition \ref{prop: torelli for very general verra with transcendental}) or \textit{(b)} (Theorem \ref{thm: derived torelli for verra fourfolds}).
\end{proof}

\begin{remark}
    Note that Theorem \ref{thm: L implies D for Verra of type I} holds for all Verra fourfolds that satisfy derived Torelli with $T(V_4)$.

    Moreover, since very general Verra fourfolds (whose associated twisted K3 surfaces have Picard rank $\rho = 1$) satisfy the Torelli theorem (cf. Proposition \ref{prop: torelli for very general verra with transcendental}), $\bL$-equivalence implies isomorphism for these Verra fourfolds.
\end{remark}

Recall from Theorem \ref{thm: finitely many FM partners} that a Verra fourfold has finitely many Fourier--Mukai partners. Therefore, if the Fano $\bL$-equivalence conjecture holds, $\bL$-equivalence classes of Verra fourfolds should also be finite. We now show that this is indeed the case.

\begin{theorem}\label{thm: L equivalence classes are finite}
    Given a Verra fourfold $V_4$, there are finitely many Verra fourfolds that are $\bL$-equivalent to $V_4$.
\end{theorem}
\begin{proof}
    As in \cite{Efi18,Mei,MeinsmaOnLEquivalence}, it follows from the existence of the Hodge realisation morphism: 
    \[
    \functionstar{\gravc[\bL^{-1}]}{\gghsfour}{\ [X]}{[H^4(X,\bZ)]}
    \]
    that any two Verra fourfolds $V_4$ and $V_4'$ which are $\bL$-equivalent satisfy $[H^4(V_4,\bZ)] = [H^4(V_4',\bZ)] \in \gghsfour$.
    The class $[H^4(V_4,\bZ)] \in \gghsfour$ contains finitely many isomorphism classes of objects by \cite[Theorem 2.7]{Efi18}. The result now follows from Lemma \ref{lem: finite up to hodge isomorphism} below.
\end{proof}

\begin{lemma} \label{lem: finite up to hodge isometry}
     Let $V_4$ be a Verra fourfold. Then there are finitely many isomorphism classes of Verra fourfolds $V_4'$ for which there is a Hodge isometry $H^4(V_4,\bZ) \simeq H^4(V_4',\bZ)$.
\end{lemma}
\begin{proof}
    Let $V_4'$ be a Verra fourfold for which there is a Hodge isometry $H^4(V_4,\bZ) \simeq H^4(V_4',\bZ)$. Restricting this Hodge isometry to transcendental lattices, we obtain $T(V_4) \simeq T(V_4')$. 
    Therefore, it suffices to show that there are finitely many isomorphism classes of Verra fourfolds $V_4'$ for which there exists a Hodge isometry $T(V_4) \simeq T(V_4')$. 

    Let $(S,\alpha)$ be one of the two associated twisted K3 surfaces of $V_4$.
    Firstly, we note that there are finitely many isomorphism classes of K3 surfaces $S'$ that admit a Brauer class $\alpha'$ such that $T(S,\alpha) \simeq T(S',\alpha')$. Indeed, there are finitely many overlattices of index two $T(S,\alpha)\hookrightarrow T$, and for each such overlattice $T$ there are finitely many K3 surfaces $S'$ which admit a Hodge isometry $T \simeq T(S')$. 

    Secondly, each such K3 surface $S'$ admits finitely many Brauer classes $\alpha'$ for which there is a Hodge isometry $T(S,\alpha) \simeq T(S',\alpha')$, since the $2$-torsion group $\Br(S')_2 \simeq (\bZ/2\bZ)^{\oplus(22-\rho)}$ is finite. We conclude that there are finitely many twisted K3 surfaces $(S',\alpha')$ for which there is a Hodge isometry $T(S,\alpha) \simeq T(S',\alpha')$. 

    Finally, we claim that each twisted K3 surface $(S',\alpha')$ is associated with finitely many Verra fourfolds. This follows from Lemma 
    \ref{lem: finitely many double covers} and Theorem \ref{thm: kernel and cokernel of res}, since for every ample class of square $2$ on $S'$, such a Brauer class $\alpha'$ is the pullback of finitely many conic bundles $V_3 \to \bP^2$, hence there are finitely many Verra threefolds that can arise as the branch locus of a Verra fourfold $V_4'$ with $T(V_4')\simeq T(S',\alpha')$. Since a Verra fourfold is uniquely determined by its branch Verra threefold by \cite[Proposition 4.23]{JM}, this concludes the proof. 
\end{proof}

\begin{lemma}\label{lem: finite up to hodge isomorphism}
    Let $V_4$ be a Verra fourfold. Then there are finitely isomorphism classes of Verra fourfolds $V_4'$ such that there is a Hodge isomorphism (not isometry) $f\colon H^4(V_4,\bZ) \simeq H^4(V_4',\bZ)$. 
\end{lemma}
\begin{proof}
    This is completely analogous to \cite[Proposition 5.1]{MM}, where \cite[Lemma 5.2]{MM} is replaced by Lemma \ref{lem: finite up to hodge isometry}.
\end{proof}

\section{Cubic fourfolds of Verra type}

\subsection{Birationally equivalent quadric surface bundles}
\begin{theorem} \label{thm: birationality theorem}
    Let $Q_i \to \mathbf{P}^2$ be two quadric surface bundles (where $i=1,2$). Suppose that they have the same discriminant double covers $S_1 = S_2 = S \to \mathbf{P}^2$ with the same Brauer classes $\alpha_1 = \alpha_2\in \mathrm{Br}(S)_2$. Then $Q_1$ and $Q_2$ are birational.
\end{theorem}

\begin{proof}
    Let $T = \Spec(\bC(S))$ and $B = \Spec(\bC(\bP^2))$. By \cite[Theorem 1]{APL}, we have the one-to-one correspondence
    \begin{equation} \label{eq:APS correspondence}
\left\{
\begin{array}{c}
\text{Quadric surface bundles over $B$}\\
\text{with simple degeneration along} \\
C \text{ and discriminant cover } T\to B
\end{array}
\right\}
\overset{1:1}{\longleftrightarrow}
\left\{
\begin{array}{c}
\text{Quaternion algebras}\\
\text{over } T \text{ with}\\
\text{split norm to } B
\end{array}
\right\}
\;\;\hookrightarrow\;\;
\Br(\mathbf{C}(S)),
    \end{equation}
    along with the inclusion on the right hand side. Since $\alpha_1 = \alpha_2 \in \Br(\bC(S))_2$, the quaternion algebras which map to $\alpha_1, \alpha_2$ under the inclusion are equal, which by the correspondence (\ref{eq:APS correspondence}) means that the quadric bundles are equal over the generic point. In other words, for $\eta \in \bP^2$ the generic point, we have an equality of fibres $Q_{1,\eta} = Q_{2, \eta}$. But this means that that $Q_1$ and $Q_2$ are birational, as required.  
\end{proof}

\subsection{Cubic fourfolds of Verra type}

Let $X$ be a cubic fourfold containing a plane $P$. Then the projection
\[
\Bl_P X \to P
\]
is a quadric surface bundle.
The Stein factorisation of the moduli space of lines on fibres of this quadric surface bundle $\cF(\Bl_PX/P) \to P$ defines a $\bP^1$-bundle over a degree-2 K3 surface $S$. 
This way, we obtain a $2$-torsion Brauer class $\alpha\in \Br(S)_2$, and the Kuznetsov component of $X$ is $\cA_X\simeq \Db(S,\alpha)$. 
In this subsection, we will consider the special case when there is a non-trivial Brauer class $\beta \in \Br(\bP^2\setminus C)_2$ with $\res(\beta)=\alpha$, where $\res$ denotes the restriction map $\res\colon \Br(P\setminus C)_2 \to \Br(S)_2$ from Section \ref{sec: brauer classes of verra type}. 
In other words, we consider the special case that $\alpha$ is of Verra type (see Definition \ref{def: verra type}). 

\begin{definition}\label{def: cubic fourfold verra type}
    Let $X$ be a cubic fourfold containing a plane $P\subset X$. We say that the pair $(X,P)$ is \textit{of Verra type} if the twisted K3 surface $(S,\alpha)$ associated to the quadric surface bundle $\Bl_PX\to P$ is of Verra type.
\end{definition}

Before we prove the existence of cubic fourfolds of Verra type, we note that they have the following important property:

\begin{theorem}\label{thm: cubic fourfolds of verra type are birational to verra}
    Let $(X,P)$ be a cubic fourfold of Verra type, and let $(S,\alpha)$ be the associated twisted K3 surface. Then there exists a Verra fourfold $V_4$, one of whose associated K3 surfaces is also $(S,\alpha)$, which is birational to $X$:
    \[
    X\overset{\sim}\dashrightarrow V_4.
    \]
    Moreover, we have $\cA_{X}\simeq \cA_{V_4}.$
\end{theorem}
\begin{proof}
    Since $(S,\alpha)$ is of Verra type, there exists a Verra fourfold $V_4$ such that $(S,\alpha)$ is one of the two associated twisted K3 surfaces of $V_4$. This means that $(S,\alpha)$ is the twisted K3 surface associated with the quadric surface bundle $p\colon V_4 \to \bP^2$, where $p$ is one of the two projection maps. We may thus apply Theorem \ref{thm: birationality theorem} to conclude that $\Bl_PX$, and thus $X$, is birational to $V_4$. The final claim follows from \cite[Lemma 4.2]{Kuzcub}.
\end{proof}

\begin{remark}
   The blow-up $\Bl_P X$ from above and the Verra fourfold $V_4$ are not isomorphic, since their Neron--Severi lattices are not isometric. Indeed, if $\Pic \Bl_P X = \bZ \cdot h \oplus \bZ \cdot e$ where $h$ is the hyperplane class and $e$ is the class of the exceptional divisor, and $\Pic V_4 = \bZ \cdot H_1 \oplus \bZ \cdot H_2$, then we have $H_1^3H_2 = H_2^3H_1= H_1^4 = H_2^4 = 0$. Therefore, for any $a,b \in \bZ$, the number $(aH_1 + bH_2)^4 = 6a^2b^2$ is even. However, the intersection number $h^4 = 3$ is odd.
\end{remark}

We also note the following immediate corollary regarding rationality:

\begin{corollary}
    Let $X$ and $V_4$ be as in Theorem \ref{thm: cubic fourfolds of verra type are birational to verra}. Then $X$ is irrational if and only if $V_4$ is irrational.
\end{corollary}

The main goal for the remainder of this subsection is to construct examples of cubic fourfolds of Verra type.
Firstly, we note that a very general cubic fourfold containing a plane is not of Verra type.
\begin{proposition} \label{prop: generally not of verra type}
    Let $X$ be a very general member of $\cC_8$, i.e. such that $$A(X) \simeq K_8\coloneqq \begin{pmatrix}
        3 & 1 \\ 1 & 3
    \end{pmatrix}.$$ Let $P\subset X$ be a plane. Then the pair $(X,P)$ is not of Verra type. 
\end{proposition}
\begin{proof}
    Let $(S,\alpha)$ be the associated twisted K3 surface. Then we have $\rho(S) = 1$. Therefore, by \cite[\S9]{vGe}, we have an isomorphism
    \[
    \cA\colon \left(\frac{\Pic C}{K_C}\right)_2\simeq \Br(S)_2,
    \]
    where $C\subset P$ is the branch sextic of the double cover $S\to P$.  Moreover, from \cite[\S 9]{vGe}, we can see that $\alpha$ corresponds to an odd theta characteristic on $C$ through this isomorphism. If $\alpha$ were of Verra type, then there would be a $2$-torsion line bundle $\cL \in \Jac(C)_2$ with $\cA(\cL) =\alpha$ by Remark \ref{rem: compatibility of res and A}. However, since $\cA$ is injective, no such line bundle exists, so $(S,\alpha)$ is not of Verra type.
\end{proof}

On the other hand, suppose that the associated twisted K3 surface $(S,\alpha)$ satisfies
\[
\NS(S) \simeq L_{d,1}\coloneqq  \begin{pmatrix}
    2 & 1 \\ 1 & -2d
\end{pmatrix}.
\]
Then, by Theorem \ref{thm: kernel and cokernel of res}, every non-trivial $2$-torsion Brauer class in $\Br(S)_2$ is of Verra type. Thus $(X,P)$ is also of Verra type in this case. This procedure gives rise to infinitely many Noether--Lefschetz divisors of $\cC_8$ consisting of cubic fourfolds of Verra type.

We consider a cubic fourfold $X$ containing a plane $P$ such that $A(X)$ has rank $3$. 
In this set-up, it was proved in \cite{YY} that $A(X)$ is isometric to precisely one lattice of the form 
\[
M_{\tau,n} \coloneqq \begin{pmatrix}
    3 & 1 & 0 \\ 1 & 3 & \tau \\ 0 & \tau & 2n
\end{pmatrix},
\]
where $0\leq \tau \leq 4$ and $n\geq 2$. Moreover, the divisors $\cC_{M_{\tau,n}}$ of $\cC_8$ are irreducible, and non-empty if $(\tau,n)\neq (3,2),(4,2),(4,3)$.
\begin{proposition} \label{prop: verra type criterion}
    Let $X$ be a cubic fourfold containing a plane $P$ and such that $\rk \ A(X) = 3$. Pick $\tau,n$ such that there is an isometry
    \[
    A(X) \simeq M_{\tau,n}.
    \]
    Then $(X,P)$ is of Verra type if and only if $\tau = 2$.
\end{proposition}
\begin{proof}
    Let $(S,\alpha)$ be the twisted K3 surface associated with the quadric surface bundle $\Bl_PX\to P$, and let $C\subset \bP^2$ be the branch sextic of the double cover $S\to \bP^2$.

    Let us first assume $\tau = 2$.
    By \cite[Theorem 4.3.1]{GvG}, we have  
    \[
    \NS(S) \simeq L_{n-1,1}.
    \]
    Therefore, by Theorem \ref{thm: kernel and cokernel of res}, each non-trivial Brauer class in $\Br(S)_2$ is of Verra type. Since $\tau$ is even, the Brauer class $\alpha$ is non-trivial by \cite[Proposition 4.7]{Kuzcub}.
    Hence $(S,\alpha)$ is of Verra type, and therefore $(X,P)$ is of Verra type as well.

    For the converse, we show that $(X,P)$ is not of Verra type for the other four cases $\tau = 0,1,3,4$. 
    
    Firstly, if $\tau = 1,3$, then the Brauer class $\alpha$ vanishes by \cite[Proposition 4.7]{Kuzcub}. However, since $S$ is of type I by \cite[Theorem 4.3.1]{GvG}, the map $\res$ from Theorem \ref{thm: kernel and cokernel of res} is an isomorphism. Therefore is no \textit{non-trivial} Brauer class $\beta \in \Br(\bP^2\setminus C)_2$ with $\res(\beta) = \alpha = 1$. Therefore $(X,P)$ is not of Verra type. 

    On the other hand, if $\tau = 0,4$, then $\alpha$ is non-trivial. However, by Theorem \ref{thm: kernel and cokernel of res}, the cokernel of $\res$ is generated by classes of the form $\cA([\overline{\theta}])$, where $\theta$ is a theta characteristic of $C$. Since $\alpha$ is of this form, there is no Brauer class $\beta\in \Br(\bP^2\setminus C)_2$ with $\res(\beta) = \alpha$, so $(X,P)$ is not of Verra type. 
\end{proof}

\subsection{Fano varieties of lines and double EPW quartics} \label{sec: Fanos and EPWs}

As well as via their double cover description, Verra fourfolds $V_4$ can also be described as the intersection of a cone over $\bP^2 \times \bP^2$ with a quadric hypersurface. Consider the Hilbert scheme of $(1,1)$-conics in $V_4$, i.e. plane conics in $V_4$ which project to lines under the projections to each factor of $\bP^2$. Then the Hilbert scheme of such conics on a general $V_4$ (arising as an intersection with a general quadric) admits a $\bP^1$-fibration over a hyperkahler fourfold $Y$ of K$3^{[2]}$-type, called a \emph{double EPW quartic}. If $\cY$ is the moduli space of double EPW quartics, then the general member $Y \in \cY$ arises this way \cite[Theorem 0.2]{IKKR}. 

Another description of double EPW quartics is as moduli spaces of twisted sheaves on a twisted K3 surface. In particular, for the two twisted K3 surfaces $(S_1, \alpha_1)$ and $(S_2, \alpha_2)$ associated to a Verra fourfold $V_4$ via its quadric surface fibrations, by \cite[Theorem 0.3]{IKKR} we have 
\begin{equation*}
    \cM_{(0,H_1,0)}(S_1, \alpha_1) \simeq Y \simeq \cM_{(0,H_2,0)}(S_2, \alpha_2) ,
\end{equation*}
and this $Y \in \cY$ will also be general.

Now let $\cF(X)$ be the Fano variety of lines on a cubic fourfold $X$.

\begin{theorem}[{\cite[Theorem 1.4]{MSFano}}]
    Suppose that $X$ contains a plane $P$. Then there is a birational map
    \begin{equation} \label{eq: F(X) bir M}
        \cF(X) \dashrightarrow \cM_{(0,H,s)}(S, \alpha)
    \end{equation}
where $s \in \bZ$ is some integer. If furthermore, the discriminant sextic $C$ of the double cover $S \to \bP^2$ is assumed to be smooth, and the pullback of the hyperplane class in $\bP^2$ to $S$ is indecomposable (i.e. it cannot be written as the sum of two effective divisors in $S$)\footnote{The authors in \emph{loc. cit.} call a cubic fourfold satisfying these extra conditions \emph{very generic}.}, then (\ref{eq: F(X) bir M}) is either an isomorphism or a Mukai flop in the plane $P^\vee$.
\end{theorem}

\begin{corollary} \label{cor: F(X) bir Y}
    Let $X$ and $V_4$ be as in Theorem \ref{thm: cubic fourfolds of verra type are birational to verra}. Assume moreover that $\rk \ A(X) =3$, so that $A(X) \simeq M_{2,n}$ for some $n\geq 2$ by Proposition \ref{prop: verra type criterion}. Then there is a birational map
    \begin{equation*}
        \cF(X) \dashrightarrow Y .
    \end{equation*}
\end{corollary}

\begin{proof}
      Let $P\subset X$ be a plane, and let $(S,\alpha)$ be the associated twisted K3 surface of $(X,P)$. 
      Due to the description of $\NS(S)$ in our case (type I), we have $H \cdot D = 1$. Let $\cF$ be a sheaf with Mukai vector $(0,H,0)$. The divisibility condition means that $v(\cF \otimes \cO_S(D))) = (0,H,1)$. Note that since $\cO_S(D)$ is untwisted, the Brauer twist of $\cF$ does not change. Tensoring by a line bundle induces an isomorphism of moduli spaces of (twisted) Gieseker stable sheaves, hence we get the desired birational equivalence.
\end{proof}

We note that Corollary \ref{cor: F(X) bir Y} also follows from \cite[Theorem 1.1]{HM25}, see also \cite[Remark 5.3]{MM24}.

\section{The Verra involution on very general special cubic fourfolds} \label{sec: Verra involution}
In this section, we study Fourier--Mukai partners of cubic fourfolds in $\cC_{M_{2,n}}$, and study Huybrechts's conjecture and Brooke--Frei--Marquand's conjecture for very general members of $\cC_{M_{2,n}}$ for $n\geq 2$ even such that $4(n+1)-1$ is a prime power. 

The main tools used in this section are Theorem \ref{thm: birationality theorem} and the counting formula of Theorem \ref{thm: counting formula}.

\subsection{The construction}

Our main construction is as follows.

\begin{theorem} \label{thm: cubic correspondence}
    Let $(X,P)$ be a cubic fourfold of Verra type which is a very general member of $\cC_{M_{2,n}},$ where $n$ is even and $4(n-1)+1$ is a power of a prime number. Let $(S_1,H_1,\alpha_1)$ be the twisted polarised K3 surface associated with the quadric surface bundle $\Bl_P X \to P$, and let $V_4$ be a Verra fourfold, one of whose associated twisted polarised K3 surfaces is also $(S_1,H_1,\alpha_1)$. Let $(S_2,H_2,\alpha_2)$ be the other twisted K3 surface associated with $V_4$. Then there exists a unique cubic fourfold $X'$ containing a plane $P'$ such that the twisted K3 surface associated with the quadric surface bundle $\Bl_{P'} X' \to P'$ is $(S_2,H_2,\alpha_2)$:
    \begin{equation*}
        \xymatrix{
        \Bl_PX \ar[dr] \ar@{-->}[rr]^\simeq && V_4 \ar[dl] \ar[dr] && \Bl_{P'}X' \ar[dl] \ar@{-->}[ll]_\simeq \\
        & \bP^2 && \bP^2 &
        }
    \end{equation*}
\end{theorem}
\begin{proof}
    This follows immediately from Lemma \ref{lem: all for one} and Lemma \ref{lem: cubic general equivalent} below.
\end{proof}

\begin{remark} \label{rmk: not isomorphic unless n is 2}
    Keeping the notation of Theorem \ref{thm: cubic correspondence}, suppose $\NS(S) \simeq L_{n-1,1}$ for some even $n\geq 2$ and that $X$ and $X'$ are very general in the moduli space $\cC_{M_{2,n}}$. Then we claim that $X$ and $X'$ are isomorphic if and only if $n = 2$. 
    
    Firstly, since $n$ is even, $T(X)$ cannot be primitively embedded in the K3 lattice $\klat$ by \cite[Corollary 8.14(3)]{YY}. Therefore, the Verra fourfold $V_4$ satisfies the conclusions of Lemma \ref{lem: comparing k3s in rank 2}(a), i.e. we have $\nu(S_1,\alpha_1)\neq \nu(S_2,\alpha_2)$, which in particular means $(S_1,\alpha_1)\not\simeq (S_2,\alpha_2)$.
    
    Therefore, we have $(X,P) \not \simeq (X',P')$. In other words, if we have $X = X'$, then it follows that the two planes $P$ and $P'$ must be distinct, $P \neq P'$, and they cannot be disjoint, $P\cap P' \neq \emptyset$, since the quadric surface fibrations associated to two disjoint planes are isomorphic, see \cite[Remark 2.9]{StellariSingular}. Therefore, $P\cap P'$ is either a line or a single point.

    If $P\cap P' = \ell$ for some line $\ell \subset \bP^5$, then the K3 surfaces $S$ and $S'$ are elliptic, and we have 
    \[\NS(S) \simeq \begin{pmatrix} 2 & 0 \\ 0 & -2 \end{pmatrix},
    \]
    see for example \cite{HartliebPlanes}. This is a contradiction with $\NS(S) \simeq L_{d,1}$.

    On the other hand, if $P\cap P' = p$ for some point $p \in \bP^5$, then it follows that 
    $\NS(S) \simeq L_{1,1},$ in other words $n = 2$. 

    Conversely, if $n=2$, then $X$ contains two planes $P,P'$ meeting in a point, and the two associated twisted K3 surfaces are not isomorphic, so that $X = X'$.
\end{remark}

\subsubsection{Moduli spaces of twisted K3 surfaces and cubic fourfolds, after Brakkee}
We recall the construction of moduli spaces of twisted polarised K3 surfaces of \cite{Brakkee}. 

Let $h \in \klat$ be a class of square $h^2 = 2$, and write $\Lambda_2 \coloneqq h^\perp \subset \klat$. The isometry class of the lattice $\Lambda_2$ is independent of the choice of $h$, and it is given by 
\[
\Lambda_2 \simeq \langle -2 \rangle \oplus U^{\oplus 2}\oplus E_8(-1)^{\oplus2}.
\]

Given a map $w \colon \Lambda_2 \to \bZ/2\bZ$, we denote by $M_w$ the coarse moduli space of triples $(S,H,\alpha)$ consisting of
\begin{enumerate}
    \item[(1)] a K3 surface $S$;
    \item[(2)] a polarisation $H \in \NS(S)$ of degree $H^2 = 2$, and;
    \item[(3)] a map $\alpha \colon H^\perp \to \bZ/2\bZ$,
\end{enumerate}
such that there exists an isometry $\varphi \colon H^\perp \simeq \Lambda_2$ which induces a commutative diagram where the vertical maps are isomorphisms
\[
\xymatrix{
0 \ar[r] & \ker(\alpha) \ar[r] \ar[d]^\simeq & H^\perp \ar[r]^\alpha \ar[d]_\varphi^\simeq & \bZ/2\bZ \ar[r] \ar@{=}[d] & 0 \\
0 \ar[r] & \ker(w) \ar[r] & \Lambda_2 \ar[r]^w & \bZ/2\bZ \ar[r] & 0.
}
\]

Note that $\alpha \in \Hom(H^\perp, \bZ/2\bZ)$ induces a unique Brauer class via the surjective map
\begin{equation}\label{eq: not so special anymore}
\Hom(H^\perp, \bZ/2\bZ) \twoheadrightarrow \Hom(T(S),\bZ/2\bZ)\simeq \Br(S)_2. 
\end{equation}
For a K3 surface of Picard rank $\rho(S)=1$, \eqref{eq: not so special anymore} is an isomorphism. This is no longer the case for $\rho(S)>1$, meaning that any Brauer class has multiple lifts to $\Hom(H^\perp, \bZ/2\bZ)$.

For a well-chosen map, denoted $w_{1,1} \colon \Lambda_2 \to \bZ/2\bZ$ in \cite[\S2.2 and \S4.1]{Brakkee}, Brakkee constructed a birational map 
\[
\psi \colon M_{w_{1,1}}\overset{\sim}{\dashrightarrow} \cC_8
\]
which, for $\rho(S)=1$, is an inverse to the construction that associates a twisted K3 surface to a cubic fourfold with a plane. Moreover, there is an open subset $V \subset M_{w_{1,1}}$ such that $\psi|_V \colon V \simeq \cC_8$ is a biregular isomorphism. We say that a twisted polarised K3 surface contained in such a $V$ is \textit{cubic general}.

\begin{lemma}\label{lem: all for one}
    Let $V_4$ be a Verra fourfold with associated polarised twisted K3 surfaces $(S_1,H_1,\alpha_1)$ and $(S_2,H_2,\alpha_2)$. Suppose $\NS(S_1) \simeq \NS(S_2) \simeq L_{d,1}$, where $4d+1$ is a power of a prime. Assume moreover that $\alpha_1$ admits a lift $\tilde{\alpha_1} \in \Hom(H_1^\perp,\bZ/2\bZ)$ such that $(S_1,H_1,\tilde{\alpha_1}) \in M_{w_{1,1}}$.     
  Then $\alpha_2$ also admits a lift $\tilde{\alpha_2} \in \Hom(H_2^\perp, \bZ/2\bZ)$ such that the triple $(S_2,H_2,\tilde{\alpha_2})$ is contained in $M_{w_{1,1}}$.
\end{lemma}
\begin{proof}
    From the assumption that $(S_1,H_1,\tilde{\alpha_1}) \in M_{w_{1,1}}$, it follows immediately that there is a commutative diagram
    \[
    \xymatrix{
    0 \ar[r] & \ker(\tilde{\alpha_1}) \ar[r] \ar[d]^\simeq & H_1^\perp \ar[r]^{\tilde{\alpha_1}} \ar[d]_\varphi^\simeq & \bZ/2\bZ \ar[r] \ar@{=}[d] & 0 \\
    0 \ar[r] & \ker(w_{1,1}) \ar[r] & \Lambda_2 \ar[r]^{w_{1,1}} & \bZ/2\bZ \ar[r] & 0.
    }
    \]
    
    We must find a lift $\tilde{\alpha_2}$ of $\alpha_2$ such that there is a commutative square
     \[
    \xymatrix{
    H_2^\perp \ar[r]^{\tilde{\alpha_2}} \ar[d]^\simeq & \bZ/2\bZ \ar@{=}[d]\\
    \Lambda_2 \ar[r]^{w_{1,1}} & \bZ/2\bZ.
    }
    \]
    Equivalently, we must find a lift $\tilde{\alpha_2}$ of $\alpha_2$ and an isometry $\Psi \colon H_2^\perp \simeq H_1^\perp$ fitting in a commutative square
     \[
    \xymatrix{
    H_2^\perp \ar[r]^{\tilde{\alpha_2}} \ar[d]_\Psi^\simeq & \bZ/2\bZ \ar@{=}[d]\\
    H_1^\perp \ar[r]^{\tilde{\alpha_1}} & \bZ/2\bZ.
    }
    \]
    For any isometry $\Psi\colon H_2^\perp \simeq H_1^\perp$, the map $\tilde{\alpha_1}\circ\Psi:H_2^\perp \to \bZ/2\bZ$ is a lift of $\alpha_2$ if and only if $\ker(\tilde{\alpha_1}\circ\Psi) \cap T(S_2) = T(S_2,\alpha_2)$, since $\alpha_2$ is a $2$-torsion Brauer class.
    In other words, it suffices to find $\Psi$ such that $\Psi(T(S_2,\alpha_2)) = T(S_1,\alpha_1)$ and $\Psi(T(S_2)) = T(S_1)$. Therefore, the result follows from Lemma \ref{lem: lattice salads}.
\end{proof}

\begin{lemma}\label{lem: cubic general equivalent}
    Keeping the notation of Lemma \ref{lem: all for one}, $(S_1,H_1,\tilde{\alpha_1})$ is cubic general if and only if $(S_2,H_2,\tilde{\alpha_2})$ is cubic general. 
\end{lemma}
\begin{proof}
    The polarised twisted K3 surface $(S_1,H_1,\tilde{\alpha_1})$ is cubic general if and only if there is a primitive embedding $H_1^\perp \hookrightarrow \Lambda_{\mathrm{cub}}$ such that $T(S_1)^\perp \subset \Lambda_{\mathrm{cub}}$ contains no vectors of square $2$, nor any vectors of square $6$ and divisibility $3$ in $\Lambda_{\mathrm{cub}}$. Therefore, the result follows from the existence of a commutative square of (not necessarily Hodge) metric morphisms:
    \[
    \xymatrix{
    T(S_1) \ar@{^(->}[r] \ar[d]^\simeq & H_1^\perp \ar[d]^\simeq \\
    T(S_2) \ar@{^(->}[r] & H_2^\perp.
    }
    \]
    
\end{proof}

\begin{proposition} \label{prop: properties of Verra involution}
    Suppose $X$ is a very general member of $\cC_{M_{2,n}}$ for $n \neq 2$ even, and let $X'$ be the cubic fourfold constructed in Theorem \ref{thm: cubic correspondence}.
    Then both of the following statements hold.
    \begin{enumerate}
        \item $X$ and $X'$ are birational;
        \item $X$ and $X'$ are Fourier--Mukai partners.
    \end{enumerate}
\end{proposition}
\begin{proof}
    It is obvious from the construction that $X$ and $X'$ are birational. Moreover, they are Fourier--Mukai partners because we have equivalences
    \[
    \cA_X \simeq \Db(S_1,\alpha_1) \simeq \cA_{V_4} \simeq \Db(S_2,\alpha_2) \simeq \cA_{X'}.
    \]
\end{proof}

\subsection{Application to conjectures of Huybrechts and Brooke--Frei--Marquand}

In this subsection, we prove Huybrechts' conjecture and Brooke--Frei--Marquand's conjecture for cubic fourfolds in $\cC_{M_{2,n}}$ for special values of $n$. Firstly, let us recall the statements of the conjectures. 

\begin{conjecture}[{\cite[Conjecture 2.5]{HuybrechtsUpdate}}] \label{conj: Huybrechts birational categorical Torelli}
    Let $X$ and $X'$ be smooth cubic fourfolds, and suppose that their Kuznetsov components are equivalent. Then $X$ and $X'$ are birationally equivalent. In other words,
    \begin{equation*}
        \cA_X \simeq \cA_{X'} \implies X \sim_{\mathrm{bir}} X' .
    \end{equation*}
\end{conjecture}

\begin{conjecture}[{\cite[Conjecture 1.2]{BFM}}] \label{conj: BFM}
    Suppose $X$ and $X'$ are smooth cubic fourfolds with birationally equivalent Fano varieties of lines. Then $X$ and $X'$ are birationally equivalent. In other words,
    \begin{equation*}
        \cF(X) \sim_{\mathrm{bir}} \cF(X') \implies X \sim_{\mathrm{bir}} X' .
    \end{equation*}
\end{conjecture}

Firstly, we note the following fact about Fourier--Mukai partners of cubic fourfolds contained in $\cC_{M_{2,n}}$.

\begin{lemma}\label{lem: what happens in CM2n stays in CM2n}
    Let $X$ be a cubic fourfold with $A(X) \simeq M_{2,n}$. Then, if $X'$ is a Fourier--Mukai partner of $X$ containing a plane, we have $A(X') \simeq M_{2,n}$.
\end{lemma}
\begin{proof}
    Since $X$ and $X'$ are Fourier--Mukai partners, there is a Hodge isometry $T(X) \simeq T(X')$, hence we have $\disc \ T(X) = \disc \ T(X')$. By assumption, we have $\disc \ T(X) = \disc \ M_{2,n} = 16n-12$. Moreover, there exist $0\leq \tau'\leq 4,n'\geq 2$ such that $X' \in \cC_{M_{\tau',n'}}$. Then we have 
    \[
    16n-12 = \disc \ T(X) = \disc \ T(X') = 16n'-3(\tau')^2.
    \]
    Computing modulo 16, it is easy to see that it follows that we have $\tau'=2$ and thus $n'=n$. 
\end{proof}

\begin{theorem}\label{thm: FM count for CM2n}
   
    Let $(X,P)$ be a cubic fourfold of Verra type contained in $\cC_{{M}_{2,n}}$, where $n\neq 2$ is even, and $4(n-1)+1$ is a prime power, and suppose $X$ is very general with this property. Then 
    \[
    |\FM_8(X)| = 2.
    \]
    Here, $\FM_8(X)$ denotes the set of isomorphism classes of Fourier--Mukai partners $X'$ of $X$ containing a plane, i.e. the set of isomorphism classes of Fourier--Mukai partners contained in $\cC_8$.
    
    Moreover, the unique non-trivial Fourier--Mukai partner $X' \in \FM_8(X)$ is birational to $X$, and the Fano varieties of lines $\cF(X)$ and $\cF(X')$ are birational.
\end{theorem}
\begin{proof}  
   Let $V_4$ be the Verra fourfold which is birational to $X$ via Theorem \ref{thm: cubic correspondence}. Letting $(S_1,H_1,\alpha_1)$ and $(S_2,H_2,\alpha_2)$ be the associated twisted polarised K3 surfaces of $V_4$, we have $\NS(S_1) \simeq L_{n-1,1}$ by \cite[Theorem 4.3.1]{GvG}. Without loss of generality, we assume that the associated twisted polarised K3 surface of $(X,P)$ is $(S_1,H_1,\alpha_1)$. 
   
   Since $4(n-1)+1$ is a prime power, it follows from Corollary \ref{cor: categorical torelli for rank 2} that $V_4$ has no non-trivial Fourier--Mukai partners. In particular, for any twisted K3 surface $(S,\alpha)$ with $\Db(S,\alpha) \simeq \Db(S_1,\alpha_1)$, we have $(S,\alpha) \simeq (S_1,\alpha_1)$ or $(S,\alpha) \simeq (S_2,\alpha_2)$. 
   Therefore, if $X' \in \cC_{M_{2,n}}$ is a non-trivial Fourier--Mukai partner of $X$, then the associated twisted polarised K3 surface of $X'$ must be $(S_2,H_2,\alpha_2)$. There is a unique such cubic fourfold, which is birational to $X$ by Theorem \ref{thm: cubic correspondence}.

   By Corollary \ref{cor: F(X) bir Y}, we get birational equivalences $\cF(X) \dashrightarrow Y$ and $\cF(X') \dashrightarrow Y$, where $\cM_{(0,H_1,0)}(S_1, \alpha_1) \simeq Y \simeq \cM_{(0,H_2,0)}(S_2, \alpha_2)$. Taking the inverse of the second birational equivalence, we get the desired birational equivalence $\cF(X) \dashrightarrow \cF(X')$.
\end{proof}       

\begin{corollary}
    Conjecture \ref{conj: Huybrechts birational categorical Torelli} and Conjecture \ref{conj: BFM} hold for very general members of $\cC_{M_{2,4}}$.
\end{corollary}
\begin{proof}
    One can verify using the Oscar function
    \texttt{genus\_representatives()}, 
    developed by Simon Brandhorst and Stevell Muller, that $M_{2,4}$ is unique in its genus \cite{OSCAR, OSCAR-book}. 
    Let $X$ be a cubic fourfold with $A(X) \simeq M_{2,4}$, and suppose $X$ is very general with this property. Then any Fourier--Mukai partner $X'$ of $X$ also satisfies $A(X') \simeq M_{2,4}$, since $M_{2,4}$ is unique in its genus. 
    Since $4\cdot(4-1)+1 = 13$ is prime, Theorem \ref{thm: FM count for CM2n} applies to $X$, hence $|\FM_8(X)| = |\FM(X)| = 2$, and $X$ is birational to its unique non-trivial Fourier--Mukai partner, and their Fano varieties are birational as well.
\end{proof}

\bibliographystyle{alpha}
{\small{\bibliography{references}}}
\end{document}